\documentclass{amsart}
\usepackage{graphicx} 
\usepackage{amsmath}
\usepackage{amssymb}
\usepackage{amsthm}

\usepackage{tikz}
\usepackage{enumerate} 
\usepackage{hyperref}
\usepackage[dvipsnames]{xcolor}

\newtheorem{theorem}{Theorem}[section]
\newtheorem{lemma}[theorem]{Lemma}
\newtheorem{proposition}[theorem]{Proposition}
\newtheorem{corollary}[theorem]{Corollary}
\newtheorem{question}[theorem]{Question}
\newtheorem{claim}[theorem]{Claim}

\newtheorem{introtheorem}{Theorem}

\newtheorem{introcor}{Corollary}

\theoremstyle{definition}
\newtheorem{definition}[theorem]{Definition}
\newtheorem{assumption}[theorem]{Assumption}

\theoremstyle{remark}
\newtheorem{remark}{Remark}[theorem]

\definecolor{mypurple}{HTML}{A205B3}
\definecolor{myorange}{HTML}{FD7121}
\definecolor{myblue}{HTML}{008DB8}
\definecolor{mygreen}{HTML}{5FB800}
\definecolor{menger1}{HTML}{00e9d8}
\definecolor{menger2}{HTML}{ffe000}
\definecolor{menger3}{HTML}{ff00a1}

\title[Menger Curve and Non-Planar Boundaries of Quotients]{Menger Curve and Non-Planar Boundaries in Quotients of Hyperbolic Groups}
\author{Rafaela Ioannou}
	\address{Maxwell Institute and Department of Mathematics, Heriot-Watt University,     Edinburgh, UK}
	\email{ri4005@hw.ac.uk}

\newcommand{\G}{\Gamma}
\newcommand{\V}{\mathcal{V}}
\newcommand{\Psub}{\mathcal{P}}
\newcommand{\diam}{\mathrm{diam}}
\newcommand{\e}{\varepsilon}
\newcommand{\la}{\lambda}
\newcommand{\bdry}{\partial}
\newcommand{\gbar}{\hat{G}}
\newcommand{\gbardef}{\hat{G}=G\big/ \langle\langle g^n \rangle\rangle}

\newcommand{\normal}{\unlhd}

\begin{document}
\begin{abstract}
We study the persistence of non-planarity in boundaries of hyperbolic groups under quotients, focusing on quotients by large enough powers of infinite order elements. We prove that the non-planarity of the boundary of a one-ended hyperbolic group is preserved under such quotients. As a consequence, we obtain an analogous result about the persistence of Menger curve boundaries. 
%
\end{abstract}
\maketitle
\section{Introduction}
The Menger curve comes up naturally when discussing topological spaces as boundaries of hyperbolic groups. Indeed, by \cite[Theorem 4]{kapovichkleiner} if a group $G$ is hyperbolic with one-dimensional boundary and does not split non-trivially over any elementary subgroup, then one of the following holds:
  \begin{enumerate}
   \item The boundary $\bdry G$ is a Menger curve,
   \item $\bdry G$ is a Sierpi{\'n}ski carpet,
   \item $\bdry G$ is homeomorphic to the circle $\mathbb{S}^1$.
  \end{enumerate}
  Out of the three options, only the Menger curve is non-planar, meaning it does not embed as a subset of the sphere. Hence, if a hyperbolic group satisfies the above conditions and its boundary is not planar, then it is homeomorphic to the Menger curve.
 
 Hyperbolic groups like this exist in abundance; for example, a ``generic" finitely presented group (Gromov density model) is hyperbolic with Menger curve boundary \cite{dgp}. In fact, Menger curves naturally appear as boundaries of groups arising in a number of settings in addition to random groups: examples include small cancellation groups \cite[Theorem 4.18]{champetier1995proprietes}, hyperbolic right-angled Coxeter groups \cite{danielski2022right}, or Bourdon buildings \cite{benakliPolyedresAvecNombre1993,dymaraBoundariesRightangledHyperbolic2007}.
 
 The work in this paper is motivated by the following question: does the property of having a Menger curve boundary pass to hyperbolic quotients? As a step toward this, we study quotients by cyclic subgroups. We prove the following, which follows from a more general result on the topology of the boundary (see Theorem \ref{maintheorem} below):
 \begin{introtheorem}\label{maincorollary}
  Let $G$ be a hyperbolic, torsion-free group, such that the Gromov boundary $\bdry G$ is a Menger curve. Let $g\in G$ be an infinite order element. Then, for all large enough $n$, $\gbar = G\big/ \langle\langle g^n \rangle \rangle$ is hyperbolic and has Menger curve boundary. 
 \end{introtheorem}
A more technical version of the theorem covers groups with torsion, see \ref{maincortorsion}.

 Such quotients have been well-studied, for example, their hyperbolicity was shown in \cite{delzant, olshanskii, gromov1987hyperbolic}. To show that $\bdry \gbar$ is non-planar, it suffices to show that there exists a non-planar graph embedded in the boundary. In turn, to show that a graph is non-planar, it suffices to show that it contains either the complete bipartite graph $K_{3,3}$, or the complete graph $K_5$ (Figure \ref{fig:K33}) as a minor; meaning that we can obtain the $K_{3,3}$ or $K_5$ by contracting edges, deleting vertices or edges of the original graph. This follows directly from a classical result in graph theory called Kuratowski's Theorem \cite{Kuratowski1930}. The Menger curve is \textit{universal}, and as a consequence contains every graph, in particular both $K_{3,3}$ and $K_5$ \cite{menger}.
\begin{figure}
\centering
\scalebox{1.4}{ \begin{tikzpicture}
    \draw (0,-1) -- (2,-1);
     \draw (0,-1) -- (2,0);
      \draw (0,-1) -- (2,1);
       \draw (0,0) -- (2,-1);
     \draw (0,0) -- (2,0);
      \draw (0,0) -- (2,1);
       \draw (0,1) -- (2,-1);
     \draw (0,1) -- (2,0);
      \draw (0,1) -- (2,1);
            \filldraw[myblue] (0,-1) circle (1.42pt);
            \filldraw[myblue] (0,0) circle (1.42pt);
            \filldraw[myblue] (0,1) circle (1.42pt);
            \filldraw[myblue] (2,-1) circle (1.42pt);
            \filldraw[myblue] (2,0) circle (1.42pt);
            \filldraw[myblue] (2,1) circle (1.42pt);
    \end{tikzpicture}
    }
\qquad
\begin{tikzpicture}
\draw (0,0) -- (2,0);
\draw (0,0) -- ({2*cos(108)}, {2*sin(108)});
\draw (2,0) -- ({2+2*cos(72)}, {2*sin(72)});
\draw ({2*cos(108)}, {2*sin(108)}) -- (1,{2*((sin(108)+sin(36))});
\draw ({2+2*cos(72)}, {2*sin(72)}) -- (1,{2*((sin(108)+sin(36))});
\draw (1,{2*((sin(108)+sin(36))}) -- (0,0);
\draw (1,{2*((sin(108)+sin(36))}) -- (2,0);
\draw (0,0) -- ({2+2*cos(72)}, {2*sin(72)});
\draw (2,0) -- ({2*cos(108)}, {2*sin(108)});
\draw ({2*cos(108)}, {2*sin(108)}) -- ({2+2*cos(72)}, {2*sin(72)});
\filldraw[myblue] (0,0) circle (2pt);
\filldraw[myblue] (2,0) circle (2pt);
\filldraw[myblue] ({2*cos(108)}, {2*sin(108)}) circle (2pt);
\filldraw[myblue] ({2+2*cos(72)}, {2*sin(72)}) circle (2pt);
\filldraw[myblue] (1,{2*((sin(108)+sin(36))}) circle (2pt);
\end{tikzpicture}
    \caption{The graphs $K_{3,3}$ and $K_5$.}
    \label{fig:K33}
\end{figure}
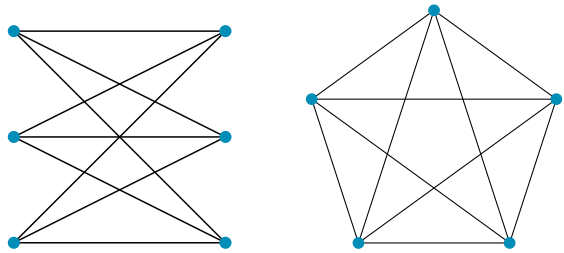
With that in mind, our main, more general result on the topology of the boundary gives a way of preserving the non-planarity of a graph under taking the quotient:
\begin{introtheorem}\label{maintheorem}
    Let $G$ be a torsion-free hyperbolic group. 
    Let $\G$ be a finite simplicial graph of valence at most 4, such that $\G$ embeds into $\bdry G$. Then for all infinite order elements $g\in G$, for all sufficiently large $n$, we have that $\gbar=G\big/ \langle\langle g^n\rangle\rangle$ is hyperbolic, 
    and there exists a graph $\hat{\G}$ that topologically embeds into $\bdry \gbar$ and contains $\G$ as a minor.
\end{introtheorem}

As with Theorem \ref{maincorollary}, a more technical version of this is stated in Theorem \ref{maintheoremtorsion} and covers groups with torsion.

In fact, the above theorem does not only imply non-planarity in the case of the boundary $\bdry G$ being a Menger curve. Indeed, as long as $G$ is one-ended, its boundary has no global cut points \cite{bowditch1998cut,swarup1996cut}. By \cite[Page 835]{claytor1934topological}, a locally connected boundary with no cut points is non-planar if and only if it contains a non-planar graph. We obtain that if $G$ has non-planar boundary, so does $\gbar$:

\setcounter{introcor}{2}
\begin{introcor}\label{generalnonplanar}
Let $G$ be a one-ended, torsion-free hyperbolic group with non-planar boundary. Let $g\in G$ be an infinite order element. Then, for all large enough $n$, the hyperbolic group $\gbar = G\big/ \langle\langle g^n\rangle\rangle$ also has non-planar boundary.
\end{introcor} 

Again, a more technical version of this result also holds for torsion groups, see Corollary \ref{corctorsion}.

Returning to the Menger curve, it remains to show that if $\bdry G$ is one-dimensional, so is $\bdry \gbar$. The dimension of the boundary is related to cohomological properties of the group, by work of Bestvina and Mess in \cite{bestvinamess}. The quotient group can be realised as a group theoretic Dehn filling, the cohomological properties of which have been studied by Wang in \cite{wang2021spectral} and Petrosyan and Sun in \cite{petrosyan2024cohomology}.

If $G$ is torsion-free, we could deduce one-dimensionality from either paper. For technical reasons, the results from \cite{petrosyan2024cohomology} do not allow us to conclude one-dimensionality in general, so we use \cite{wang2021spectral}.

This allows for an even more general version of the main result, allowing for $G$ to have torsion; we refer the reader to the results in Section \ref{sec:proofs}.
\subsection{Open questions}
We have essentially shown that, if $\bdry G$ is a Menger curve, the groups $G$ and $\gbar$ as in Theorem \ref{maincorollary} have homeomorphic boundaries, which could suggest a quasi-isometry between the groups, as no standard invariant can differentiate between them:
\begin{question}
    Can the groups $G$ and $\gbar$ in Theorem \ref{maincorollary} be quasi-isometric to each other for infinitely many $n$?
\end{question}
Another question the paper poses is regarding the non-planarity of the quotient boundary. If $\bdry G$ is homeomorphic to, say, the Sierpi{\'n}ski carpet, there is no suggestion in this paper that $\bdry \gbar$ is also planar.
\begin{question}
    Let $G$ be a hyperbolic, torsion-free group with one-dimensional boundary, and let $\gbar$ be defined as in Theorem \ref{maincorollary}. When does the quotient group $\gbar$ have non-planar boundary? 
\end{question}
Finally, one could hope to generalise this work to apply to a wider array of quotient groups:
\begin{question}
    Do Theorems \ref{maincorollary} and \ref{maintheorem} still hold for $\gbar = G\big/ \langle \langle g \rangle \rangle$, where the element $g$ satisfies some small cancellation condition?
\end{question}
The machinery of Dehn filling pairs, that we use to assert the lack of elementary splittings in the quotient group, does not apply to this choice of quotient, however this does not necessarily imply that a version of the theorems in this paper cannot hold. As a starting point, one could generalise a version of Wang's results in \cite{wang2021spectral} for this choice of $\gbar$, hence asserting the one-dimensionality of the quotient boundary.

\subsection{Outline}
Section \ref{sec:2} asserts the lack of elementary splittings of $\gbar$. To achieve this, we view the theorem's setup as a group-theoretic Dehn filling; this allows the use of machinery on the behaviour of elementary splittings under Dehn fillings from \cite{grovesmanning}.

 To ensure that a graph embedded in $\bdry G$ appears in $\bdry \gbar$, it must remain sufficiently far from points in $\bdry G$ represented by geodesic rays which don't survive the quotient. Namely, these are the orbit of the limit points of the infinite order element $g$ in $\bdry G$. The aim of Section \ref{quasiarcssection} is to obtain a way to shift arcs away from a collection of ``well-behaved" obstacles. To achieve this we make use of results from \cite{mackaysisto}. A significant part of this section is to ensure that we have substantial control of the resulting shifted arc. 

In Section \ref{sec:obstacles} we verify that the work done in the previous section can be applied to arcs in $\bdry G$. This ensures that we can shift arcs in $\bdry G$ away from the obstacles given by the orbit of limit points of $g$.

In Section \ref{bdrymapsection} we construct a ``partial map" between the two boundaries $\bdry G$ and  $\bdry \gbar$, and show that graphs which ``stay sufficiently far" from the obstacles introduced above will appear in the quotient boundary. 

The crux of the proof, and our main contribution, can be found in Section \ref{sec:6}. First, we show that we can shift any graph away from the collections of obstacles in $\bdry G$, 
while keeping any resulting intersections within arbitrarily small neighbourhoods around the vertices. The main result of the section, Proposition \ref{tren}, proves that, given any finite simplicial graph $\G$ of valence at most 4, we can shift it to find a graph $\G'$ which not only avoids the collection of obstacles, but also contains $\G$ as a minor. The main technique we use is to control the intersections, which are contained within the neighbourhoods around the vertices, and minimise them, finding new vertices and edges included in the shifted version of the original graph.

The proofs of the main results can be found in Section \ref{sec:proofs}. The theorems stated in this section are generalised in a way that allows the group $G$ to have torsion, however, torsion restricts our choice of $n$ in Theorems \ref{maincorollary} and \ref{maintheorem} and Corollary \ref{generalnonplanar}. 
\subsection{Notation} 
Throughout this paper, letting $G$ be a group, $\langle\langle g \rangle \rangle$ denotes the normal closure in $G$ of an element $g\in G$. Where $G$ is a hyperbolic group, $\bdry G$ denotes the Gromov boundary of $G$.

For a finitely generated group $G$, when we say $d_{G}(\cdot,\cdot)$ we mean the distance in the Cayley graph of $G$ for a fixed finite generating set $S$.

In a metric space $Z$, $B(z,n)$ denotes the ball of radius $n$ around an element $z\in Z$. For a set $X \subseteq Z$, $N(X,n)$ denotes the $n$-neighbourhood around $X$,

Given an arc $J$ in some metric space $(Z,\rho)$, $J[x,y]$ denotes the subarc of $J$ with endpoints $x,y\in Z$.
\subsection{Acknowledgments} The author would like to thank Nansen Petrosyan and Bin Sun for very helpful conversations on group cohomological dimension, as well as Matthew Cordes and Alessandro Sisto for their constructive comments and suggestions. 

\section{Splittings over elementary subgroups}\label{sec:2}
The behaviour of elementary splittings under group-theoretic Dehn fillings has already been studied by Groves and Manning in \cite{grovesmanning}; we shall outline their relevant results which can be applied to our current setting. Let $G$ be a hyperbolic group and $g\in G$ an infinite order element. Let $E(g)$ denote the unique, maximal elementary subgroup containing the hyperbolic element $g$, see Theorem \ref{elementarysubgroup}. The main result of this section is Corollary \ref{nosplitbdry}, which states that if a hyperbolic group $G$ does not split over $E(g)$ for any infinite order element $g\in G$, for sufficiently large $n$, neither does $\gbardef$. This is done by viewing the quotient group $\gbar$ as a group-theoretic Dehn filling of $G$.  

A somewhat surprising setting we use for both this section as well as some results in Section 4 is relative hyperbolicity. Even though $G$ and $\gbar$ are both hyperbolic, which is a much stronger assumption, the reason we use the machinery of relative hyperbolicity is that the tools used when studying relatively hyperbolic groups are particularly useful in this setting, since it allows us to study the hyperbolic properties of a \textit{pair} of a group and a collection of subgroups. We begin this section with the required results on relatively hyperbolic groups.
\subsection{Relative hyperbolicity} 
Groves and Manning's work on elementary splittings requires the group $G$ to be relatively hyperbolic with respect to a collection of subgroups $\Psub$. There are numerous definitions and approaches to studying relative hyperbolicity. We refer the reader to \cite{mackaysisto2024maps} and \cite{sisto2012metric} for a more explicit study of relatively hyperbolic groups.

It turns out, $G$ is hyperbolic relative to the subgroup $E(g)$ for any $g\in G$ of infinite order. We first provide a more detailed description of $E(g)\leq G$.
\begin{theorem}\cite{osin2006elementary}\label{elementarysubgroup}
    Let $G$ be a hyperbolic group, and $g\in G$ a hyperbolic (or, equivalently, infinite order) element. Then $g$ is contained in a unique, maximal, elementary subgroup $E(g)$. In addition, $G$ is relatively hyperbolic with respect to the collection $\Psub$ containing only $E(g)$ for some infinite order element $g\in G$.
\end{theorem}
 Recall that a subgroup is called elementary if it contains a cyclic subgroup of finite index, therefore it is virtually cyclic. 

The above theorem allows for the use of results on relative hyperbolicity on the pair $(G,E(g))$. In particular, in what follows, this will allow us to deduce that $\gbar$ admits no non-trivial elementary splittings.
\subsection{Elementary splittings}
The quotient group $\gbardef$ in Theorems \ref{maincorollary} and \ref{maintheorem} can be realised as a \textit{Dehn filling} of the group $G$, a group-theoretic analogue of Dehn fillings in 3-manifolds. Splittings of groups under Dehn fillings have been studied, and non-trivial elementary splittings persist under the quotient. We begin by introducing the concept of group fillings.
\begin{definition}\cite{osinPeripheralFillingsRelatively2007,grovesDehnFillingRelatively2008}\label{def:dehn}
Suppose that $G$ is a group and $\Psub = \{ P_1 ,\ldots , P_n \}$ is a collection of subgroups.  A Dehn filling (or just filling) of $(G,\Psub)$ is a quotient map: $\pi \colon  G \to G/K$, where $K$ is the normal closure in $G$ of some collection $K_i \unlhd P_i$.  We write
 \[     G/K = G(K_1, \ldots, K_n)       \]
 for this quotient.   We also write $\pi\colon  (G,\Psub)\to (\bar{G},\bar{\Psub})$, where $\Psub$ is the collection of images of the $P\in \Psub$.

 We say that a property holds {for all sufficiently long fillings of $(G,\Psub)$} if there is a finite set $\mathcal{B} \in G \setminus \{ 1 \}$ so that whenever $K_i \cap \mathcal{B} = \emptyset$ for all $i$, the group $G/K$ has the property.
\end{definition}
\begin{remark}\label{dehnsetuptf}
  The setup in Theorems \ref{maincorollary} and \ref{maintheorem} can be realised as a Dehn filling. To see this, take $G$ to be a torsion-free hyperbolic group and $g\in G$ an element of infinite order. Take $\Psub$ to be the collection consisting of only one subgroup, namely $E(g)$, and $K$ to be the collection consisting of the subgroup $\langle g^n \rangle \leq E(g)$ for all large enough $n$. We claim that $\langle g^n\rangle$ is a normal subgroup of $E(g)$. Since the group $G$ is torsion-free, so is $E(g)$. Since $E(g)$ is elementary, this implies that, in particular, it is cyclic, and it follows that $\langle g^n\rangle \unlhd E(g)$. Hence, the quotient map $\pi\colon  G \to \gbar$ is a Dehn filling, where $\gbar=G\big/ \langle\langle g^n \rangle \rangle$.
\end{remark}
We can realise our setup as a group-theoretic Dehn filling regardless of the assumption of $G$ being torsion-free. However, doing so will require us to lose generality in our assumptions for $n$ in $\gbardef$:
\begin{remark}\label{dehnsetup}
A version of the above remark also holds when we drop the torsion-free assumption for $G$. Since $E(g)$ might not be cyclic, $\langle g^n \rangle$ might not be a normal subgroup in $E(g)$. However, there exists $k> 0$ such that for all multiples $n$ of $k$, it is a normal subgroup. To see this, consider $\langle g^k \rangle$. If $\langle g \rangle \normal E(g)$, then so is $\langle g^k\rangle$ for any $k$. If not, it is a standard result that there exists some $k$ such that $\langle g^k\rangle\unlhd E(g)$. Therefore, setting $n$ in the setup of Theorems \ref{maincorollary} and \ref{maintheorem} such that $n$ is a multiple of $k$, we have that $\langle g^n\rangle \unlhd E(g)$. Therefore, we can conclude that there exists some $n$ such that $\gbar = G\Big/ \langle\langle g^n \rangle \rangle$ is a Dehn filling of $G$.
\end{remark}
\begin{definition}
    We say $G$ admits an elementary splitting if $G$ splits over an elementary subgroup $E\leq G$.
\end{definition}
\begin{theorem}\cite[Theorem 1.1]{grovesmanning}\label{gmsplit}
Suppose that $(G,\Psub)$ is relatively hyperbolic, with $\Psub$ consisting of virtually polycyclic subgroups. If G does not admit any non-trivial elementary splittings, then sufficiently long Dehn fillings of $G$ do not admit any non-trivial elementary splittings.
\end{theorem}
\begin{corollary}\label{nosplitbdry}
Let $G$ be a hyperbolic group that admits no non-trivial elementary splittings. Let $g\in G$ be an infinite order element. Then, there exists $k$ such that for all multiples $n$ of $k$, $\gbar=G\big/ \langle\langle g^n\rangle\rangle$ does not admit any non-trivial elementary splittings. If, in addition, $G$ is torsion free, the above holds for all sufficiently large $n$.
\end{corollary}
\begin{proof}
To see this, we use Theorem \ref{gmsplit}. Taking $\Psub$ to be the collection consisting of $E(g)$, $(G,\Psub)$ is relatively hyperbolic by Theorem \ref{elementarysubgroup}. Moreover, $\Psub$ contains only $E(g)$, which is virtually cyclic, therefore virtually polycyclic. By assumption, $G$ admits no non-trivial elementary splittings, hence by Theorem \ref{gmsplit}, sufficiently long Dehn fillings of $G$ do not admit non-trivial elementary splittings. Hence, by Definition \ref{def:dehn}, there are finitely many $k$ that are not allowed in $\langle g^k \rangle$; we refer to this finite collection as $\mathcal{K}$. Thus, let $k$ be as in Remark \ref{dehnsetup}: in particular, $k$ is chosen such that for all multiples $n$ of $k$, $\langle g^n \rangle\unlhd E(g)$. If $k\in \mathcal{K}$ instead choose $k'$ to be a multiple of $k$. There are only finitely many integers in $\mathcal{K}$, so there exists a multiple for which $\gbar$ admits no non-trivial elementary splittings. 

If $G$ is torsion-free, then by Remark \ref{dehnsetuptf} $\langle g^n \rangle\unlhd E(g)$ for any $n$. Again, the property fails for finitely many choices of $n$, hence picking $n$ to be large enough guarantees that $\gbar$ admits no-non trivial elementary splittings. 
\end{proof}
\section{Quasi-arcs that avoid obstacles}\label{quasiarcssection}
We turn our focus on arcs within a metric space. We aim to find a way to move arcs to ensure they avoid a collection of separated obstacles, as in \cite{mackaysisto}. Indeed, in order for a subset of the boundary to survive the quotient, it must avoid all points in $\bdry G$ which ``disappear'', meaning they correspond to geodesic rays in $G$ that are mapped to bounded sets in the quotient group (see Section \ref{bdrymapsection} for a more detailed discussion of what it means for a subset of the boundary to survive the quotient). The results in this section show that, given any arc in a sufficiently ``nice" metric space, we can shift it to avoid a collection of obstacles while staying close to the original arc and controlling the diameter. More specifically, the purpose of this section is to prove Lemma \ref{quasiarcs}, which elaborates on the arguments in the aforementioned publication. We begin by recalling the machinery of Mackay and Sisto.
\subsection{Basic definitions} Before we precisely state the upcoming results, we provide some necessary definitions. Unless stated otherwise, the discussion that follows uses terminology from \cite{mackay}.

Firstly, the space that we work in must be \textit{doubling}. Letting $(Z,\rho)$ be a metric space, we say it is $N$-doubling if there exists a constant $N$ such that, given any ball in $Z$, we can cover it by no more than $N$ balls of half its radius. Moreover, we say $Z$ is $L$\textit{-linearly connected} for a constant $L\geq 1$, if any two points $x,y \in Z$ can be connected by a compact, connected set $J$ such that $\diam(J)\leq Ld(x,y)$. 

A \textit{quasi-arc} is a stronger notion than that of an arc. More precisely, letting $J$ be an arc in some $N$-doubling, $L$-linearly connected metric space $Z$, we say it is $\la$-quasi-arc if for all $x,y\in J$, the diameter of the subarc with endpoints $x,y$ satisfies $\diam(J[x,y])\leq \la d(x,y)$. $J$ is an $\e$-local $\la$-quasi-arc if, for all $x,y\in J$ with $\rho(x,y)\leq \e$, the diameter of the subarc with endpoints $x,y$ satisfies $\diam(J[x,y])\leq \la d(x,y)$.

The next definitions make the notion of ``avoiding sets" more explicit. First, we define sets which are \textit{porous}, meaning, intuitively, that we can linearly move away from them.

\begin{definition} \cite[Definition 6.8]{mackaysisto} 
    Let $(Z,\rho)$ be a metric space, and $V\subseteq Z$. We say $V$ is $C$-porous on scales below $\delta$ if for any $z\in V$ and $0<r<\delta$, there exists some element $z'\in B(z,r)$ such that $\rho(z',V)\geq r/C$.
\end{definition}
When shifting arcs to avoid obstacles, it is useful to have some control over the resulting arc, in the sense that it ``stays sufficiently close" to the original arc. We use a definition which is stronger than that of an arc being contained within a neighbourhood of another:
\begin{definition}\cite{mackay}
    Let $J,J'$ be arcs in a metric space $(Z,\rho)$. We say $J'$ $\e$-follows $J$ if there exists a map $p\colon J'\to J$ which sends endpoints to endpoints, and for all $x,y\in J'$, $J'[x,y]\subseteq N(J[p(x),p(y)],\e)$. The map $p$ does not have to be continuous.
\end{definition}
Before moving onto the next definition, we introduce notation for the \textit{annular neighbourhood}: in a set $X \subseteq Z$, $A(X,n,m)$ denotes the annular neighbourhood around $X$ for radii $n,m$ for some $n\leq m$. 
\begin{definition}\cite[Definition 6.3]{mackaysisto}
Let $(Z,\rho)$ be a complete, connected metric space. A set $V \subseteq Z$ is $L$-avoidable on scales below $\delta$ for some $L\geq 1$ and $\delta \in (0,\infty ]$, if for any $r \in (0, \delta/2L)$, whenever there is an arc $I\in Z$ and points $x, y \in I \cap A(V, r, 2r)$ such that $I[x, y] \subset
N(V, 2r)$, there exists an arc $J \subset A(V, r/L, 2rL)$ with endpoints $x, y$ so that $J$ $(4rL)$-follows $I[x, y]$.
\end{definition}
Now, let $(Z,\rho)$ be a compact metric space. The following definition makes the notion of a ``collection of separated obstacles" more precise. First, we define a ``modified relative distance function" and call a collection of sets separated if the relative distance between any pair of sets in the collection is sufficiently large:
\begin{definition}\cite[Definition 7.1]{mackaysisto}
    Let $\V$ be a collection of compact subsets $V\subseteq Z$. Let $D\colon \V \to (0,\infty)$ be a map, which we call a scale function. We define $\Delta\colon \V \times \V \to [0,\infty)$ by $$\Delta(V_1,V_2)=\frac{\rho(V_1,V_2)}{\mathrm{min}\{D(V_1),D(V_2)\}}$$ for $V_1,V_2 \in \V$. Then, we say $\V$ is $L$-separated for some constant $L$ if for all $V_1,V_2 \in \V$ such that $V_1 \neq V_2$, $\Delta(V_1,V_2)\geq 1/L$.
\end{definition}

\subsection{Avoiding sets while controlling the quasi-arc}
We now move on to the main content of this section; shifting arcs in a metric space, to ensure that they avoid a collection of obstacles. In what follows, we shall use the following setup:
\begin{assumption}\label{assumptiongeneral}
    Let $(Z,\rho)$ be an $N$-doubling, $L$-linearly connected, compact metric space, with constants $L\geq 10$, $N\geq 1$. Let $\V$ be an $L$-separated collection of compact subsets $V\subseteq Z$ and $D\colon  \V \to (0, \infty)$ be a scale function, such that $V \in \V$ is $L$-porous and $L$-avoidable on scales below $D(V)$.
\end{assumption}
The above assumption guarantees that both the metric space that we work in, as well as the collection of obstacles, are, in a sense, well-behaved. We now define what it means for a set to avoid such a collection.
\begin{definition}\label{avoids}
    Let $(Z,\rho)$ be a metric space, and $\V$ a collection of subsets in $Z$. We say a set $X\subseteq Z$ $\alpha$-avoids $\V$ if $\rho(X,V)\geq 1/\alpha$ for all $V\in \V$. 
\end{definition}
We recall a result from Chapter 7 of \cite{mackaysisto}, which, given a pair of points in $Z$, allows us to obtain points that avoid $\V$ and are not ``too far" from the original points. The following proposition actually holds in a more general setting than our assumption, as $\V$ does not need to be $L$-avoidable.

\begin{proposition}\label{MS7.3}\cite[Proposition 7.3]{mackaysisto}
Let $(Z, \rho)$ and $\V$ satisfy Assumption \ref{assumptiongeneral}. For any $r \leq 1/5L$, given $p \in Z$ there exists $q \in Z$ so that $\rho(p,q) \leq \diam(Z) r$, and that $q$ $8L^2/r$-avoids $\V$.
\end{proposition}

The next lemma ensures that, given an arc in $Z$, even if its endpoints do not avoid $\V$, it is possible to move the endpoints slightly and obtain a new arc that remains close to the original. 

\begin{lemma}\label{endpoints}
   Let $(Z,\rho)$ and $\V$ satisfy Assumption $\ref{assumptiongeneral}$. For all $\varepsilon >0$, there exists $\alpha$ such that for all $x,y\in Z$, there exist $x',y'\in Z$ that satisfy the following:\\ For every arc $J=J[x,y]$ there exists an arc $J'=J'[x',y']$ which lies within the $\varepsilon$-neighbourhood of $J$, such that the endpoints $x',y'$ $\alpha$-avoid $\V$. 
\end{lemma}

\begin{proof}
    Let $\varepsilon>0$. By Proposition \ref{MS7.3}, for any $r\leq 1/5L$, we can find points $x',y'\in Z$ that $\alpha$-avoid $\V$, where $\alpha = 8L^2/r$. In addition, $\rho(x',x)$, $\rho(y',y)\leq \mathrm{diam}(Z)r$. Since $Z$ is $L$-linearly connected, the points $x,x'$ are connected by an arc of diameter at most $Ld(x,x')\leq \mathrm{diam}(Z)Lr$. Let this arc be denoted by $J_x$. The same argument yields an arc $J_y$ from $y$ to $y'$. Now, we have a path given by the concatenation $J_x\cup J \cup J_y$. It is a standard topological fact that we can find an arc inside any path, implying that there exists an arc $J'\subseteq (J_x\cup J \cup J_y)$. We claim that this arc lies within the $(\diam(Z)Lr)$-neighbourhood of $J$. To show this, we note that, since $\diam(J_x)$, $\diam(J_y)\leq \diam(Z)Lr$, any point $j \in J'$ satisfies $\rho(j,J)\leq \diam(Z)Lr$. Now, this is true for any $r\leq 1/5L$, in particular we can choose $r=\mathrm{min}\{\frac{\varepsilon}{\diam(Z)}, \frac{1}{5L}\}$. Therefore, we have that $J'$ lies within the $(\diam(Z)Lr)$-neighbourhood of $J$, but $r\leq \frac{\varepsilon}{\diam(Z)L}$ which implies that $\diam(Z)Lr\leq \varepsilon$, and thus $J'$ lies within the $\varepsilon$-neighbourhood of $J$.
\end{proof}
Proving our next result relies on the proof of the following theorem of \cite{mackaysisto}, which allows us to build new quasi-arcs that avoid obstacles.
\begin{theorem}\cite[Theorem 7.2]{mackaysisto}\label{ms7.2}
    Let $(Z, \rho),\V$ satisfy Assumption \ref{assumptiongeneral}. 
For any $\nu \leq 1$ there exists a constant $\lambda = \lambda(N, L, \nu)$ so that, given any two points $x, y \in Z$ that $\nu$-avoid $\V$, 
then $x$ and $y$ can be joined by a $\lambda$-quasi-arc which $\lambda$-avoids $\V$. 
\end{theorem}
Now, we build such quasi-arcs that avoid obstacles, with the addition that we control their distance from the original arc.
\begin{lemma}\label{quasiarcs}
    Let $(Z,\rho)$ and $\V$ satisfy Assumption $\ref{assumptiongeneral}$. For all $\alpha$, for all arcs $J=J[x,y]$ such that the endpoints $x,y$ $\alpha$-avoid $\V$, there exists $\lambda$ and $\lambda$-quasi-arc $J'$ such that $J'=J'[x,y]$ $\lambda$-avoids $\V$ and $\varepsilon$-follows $J$ for all $\varepsilon >0$.
\end{lemma}

Theorem \ref{ms7.2} will allow us to obtain a quasi-arc $J'$, as required for Lemma \ref{quasiarcs}, however the statement alone does not guarantee that for all $\varepsilon>0$, $J'$ will $\varepsilon$-follow the original arc $J$. In order to show this, we need the following machinery from \cite{mackay}, which is used in the original proof of Theorem \ref{ms7.2}:
\begin{proposition}\label{Mackay2.1Prop}{\cite[Proposition 2.1]{mackay}} 
Given $(Z,\rho)$ that satisfies Assumption \ref{assumptiongeneral}, there exist constants
 $s=s(L,N)>0$ and $S=S(L,N)>0$ with the following property:
 for each $\iota > 0$ and each arc $A \subset Z$, there exists an arc
 $J$ that $\iota$-follows $A$, has the same endpoints as $A$,
 and satisfies
 \begin{equation} \label{mackay2.1}
  \forall x,y \in J,\  \rho(x,y) < s\iota \implies
   \diam(J[x,y]) < S\iota 
 \end{equation}
where $J[x,y]$ denotes the subarc of $J$ with endpoints $x,y$.
\end{proposition}
The following result allows us to take limits of sequences of arcs, which is how the quasi-arc $J'$ is created in the proof of Theorem \ref{ms7.2} in \cite{mackaysisto}.
\begin{lemma}\label{Mackay2.2}{\cite[Lemma 2.2]{mackay}}
Let $(Z,\rho)$ satisfy Assumption \ref{assumptiongeneral}. Let $s,S,\varepsilon,\delta$ be fixed, positive constants satisfying $\delta \leq \mathrm{min}\{\frac{s}{4+2S},\frac{1}{10}\}$. If we have a sequence of arcs $J_0,J_1,...,J_n,...$ in $Z$ such that $\forall n\geq 1$
\begin{itemize}
    \item $J_n$ $\varepsilon\delta^n$-follows $J_{n-1}$, 
    \item $J_n$ satisfies (\ref{mackay2.1}) with $\iota=\varepsilon\delta^n$ and $s,S$ as fixed above,
\end{itemize}
then the Hausdorff limit $J=\mathrm{lim}_\mathcal{H}(J_n)$ exists, and is an $\varepsilon \delta^2$-local
 $\frac{4S+3\delta}{\delta^2}$-quasi-arc. Moreover, the endpoints of $J_n$ converge to the endpoints of $J$,
 and $J$ $\varepsilon$-follows $J_1$.
\end{lemma}
\begin{proof}[Proof of Lemma \ref{quasiarcs}]
    Let $\varepsilon>0$. Let $x,y$ be points that $\alpha$-avoid $\V$ for some $\alpha$. 
    By Theorem 7.2 of \cite{mackaysisto} (Theorem \ref{ms7.2}), there exists a $\lambda =\lambda(N,L,\alpha)$ and $\lambda$-quasi arc joining the points $x,y$ that $\frac1{\lambda}$-avoids $\V$. Call this quasi-arc $J'$.
    
    It remains to show that, for all $\varepsilon >0$, $J'$ $\varepsilon$-follows $J$. In the original proof of Theorem \ref{ms7.2} seen in \cite{mackaysisto}, they make use of the machinery from Proposition \ref{Mackay2.1Prop} and Lemma \ref{Mackay2.2}, in particular, the quasi-arc $J'$ is built as follows:
    
    First, we set constants as seen in the original proof: Let $D_0=\mathrm{sup}(D(V):V\in\V)$, and set $M\in Z$ be maximal such that $\rho(x,y)<D_0r^M/4L\nu$, where $\nu$ is as seen in Theorem \ref{ms7.2}, and $r$ will be fixed to be sufficiently small later on. 
    
    Following the inductive construction as in Theorem \ref{ms7.2}, we obtain a sequence of arcs $J_M,J_{M+1},...$ where each arc $J_{M+n}$ $(D_0r^M/4L)r^n$-follows $J_{M+n-1}$. Moreover, each arc in this sequence satisfies Equation (\ref{mackay2.1}) for constants $\iota, s',S'$, where $\iota = (D_0r^M/4L)r^n$, $s'=s/\nu$, and $S'=S$. The constants $s',S'$ depend only on $L,N$ given in Assumption \ref{assumptiongeneral}. In other words, the sequence of arcs satisfies the conditions of Lemma \ref{Mackay2.2}, with constants ``$\varepsilon$" replaced by $D_0r^M/4L$ and ``$\delta$" replaced by $r$.

    Next, we can take $J_M=J$ in the above sequence. In the original proof, $J_M=J_M[x,y]\in Z$ is taken to be an arc of diameter at most $L\rho(x,y)$. $J$ satisfies this property. To see this, note that if $\diam(J)> L\rho(x,y)$, we can set the constant $L$ to be big enough such that this does not occur. This is because if a space is $L$-linearly connected for some constant $L$, then by definition it will also be $L'$-linearly connected for any $L'\geq L$. 
    
    Now, as in the conclusion of Theorem \ref{ms7.2}, we obtain the quasi-arc $J'$ as the limit arc of this sequence by applying Lemma \ref{Mackay2.2}: with our new constants, this makes $J'$ into a $\frac{4S+3r}{r^2}$-quasi arc, the endpoints of $J'$ converge to $x,y$ and it $(D_0r^M/4L)r^2$-follows $J$. Since the above arguments are true for any $\nu$, we pick $\nu=1/\alpha$. Now, the above arguments hold for any $r$ that satisfies $r\leq 1/4\nu$ and $r \leq \mathrm{min}\left\{\frac{s'}{4+2S},\frac{1}{10}\right\}$. In particular, we can set $r$ to be small enough that $r^{M+2}\leq \frac{4L\varepsilon}{D_0}$, which implies that $J'$ $\varepsilon$-follows the original arc $J$.
    
\end{proof}
\section{Collections of obstacles in the boundary}\label{sec:obstacles}
In this section we verify that the boundary $\bdry G$, where $G$ is a hyperbolic, one-ended group, satisfies all properties of Assumption \ref{assumptiongeneral}. It is a standard fact that the boundary of a hyperbolic group is compact and doubling \cite[Theorem 9.2]{bonkschramm}. Hence, we must discuss the linear connectedness of the boundary, and ensure the results on the avoidable collection of obstacles; namely, that $\V$ is $L$-separated and each compact subset $V\in \V$ is $L$-porous and $L$-avoidable on scales below a scale function $D(V)$. These technical assumptions on the boundary are important, because in our setup we have the following subsets of the boundary we wish to avoid: since $g\in G$ is a infinite order, or equivalently, hyperbolic element, it has two limit points on the boundary, denoted $\{g^{\pm \infty}\}$. In order for a subset of $\bdry G$ to ``survive" the quotient and appear as a subset of $\bdry \gbar$, as required for Theorem \ref{maintheorem}, we must make sure that we can avoid the limit points $\{g^{\pm \infty}\}$ and their orbit in the boundary. Therefore, the next step will be to verify that the collection of these limit points is ``well-behaved" enough to apply the work of the previous sections.

\begin{definition}\label{obstacles}
    Let $G$ be a hyperbolic group, and let $g\in G$ be an infinite order element. Recall that $E(g)$ is the unique, maximal elementary subgroup of $G$ containing $g$. Let $\V$ be the collection consisting of all limit sets of left cosets of $E(g)$ in $\bdry G$. In particular, $\V$ consists of subsets of the form $\{hg^{\pm\infty}\}$ where $h\in G$. 
\end{definition}
The following Proposition is a special case of results by Mackay and Sisto. We assume that $G$ is a hyperbolic group, and $g\in G$ an infinite order element. The pair $(G,E(g))$ is relatively hyperbolic by Theorem \ref{elementarysubgroup}. The results of \cite{mackaysisto} require all peripheral subgroups to be hyperbolic themselves. In our case, the only peripheral subgroup is $E(g)$, which is hyperbolic since it is virtually $\mathbb{Z}$. Moreover, since we are working on the boundary, we fix a basepoint $w\in G$ for the equivalence classes of geodesic rays that define the Gromov boundary. 
\begin{proposition}\cite[Lemma 3.2, Proposition 6.4, Lemma 6.9]{mackaysisto}\label{lem:mssep}
    Let $G$ be a hyperbolic group, and $g\in G$ an infinite order element. Consider the relatively hyperbolic pair $(G, E(g))$, where $E(g)$ denotes the unique, maximal elementary subgroup containing $g$. Let $\V$ denote the collection of limit sets of left cosets of subgroups of $E(g)$ in $\bdry G$. For $V\in \V$, consider the left coset of $E(g)$ whose limit points make up $V$. Let $d_V$ denote its distance in $G$ from the basepoint $w$. Then, the following hold:
    \begin{itemize}
        \item there exists a constant $L_1$ 
          such that, for each distinct $V,V'\in \mathcal{V}$ with $d_V\geq d_{V'}$, we have that $\rho(V, V')\geq e^{-\e d_V}/L_1$, where $\rho$ denotes the distance on the boundary.\end{itemize}
    If, in addition, $\bdry G$ is connected and locally connected with no local or global cut points, then: 
    \begin{itemize}
        \item there exists a constant $L_2$ such that, for every $V\in \mathcal{V}$, $V \subset \bdry G$ is $L_2$-avoidable on scales below $e^{-\e d_V}$,
        \item there exists a constant $L_3$ such that, for every $V\in \mathcal{V}$, $V \subset \bdry G$ is $L_3$-porous on scales below $e^{-\e d_V}$.
    \end{itemize}
        \item  \label{propavoidable}

   %
%
\end{proposition}


To ensure the properties of connectedness, local connectedness and the absence of cut points, we appeal to the work of \cite{bowditch2012relatively,grovesmanning}, as summarised in \cite{grovesmanning}:
\begin{theorem}\cite[Corollary 7.9]{grovesmanning}\label{gm7.9}
Let $G$ be hyperbolic and not virtually Fuchsian. Then, $\bdry G$ is connected with no (local or global) cut point 
if and only if $G$ has no elementary splitting.
\end{theorem}
In fact, if the hyperbolic group $G$ is one-ended, then the boundary $\bdry G$ is always locally connected \cite{bestvinamess, bowditch1999treelike, levitt1998non, swarup1996cut}. In our setup, $G$ is one-ended by Stallings theorem; we saw in Corollary \ref{nosplitbdry} that $G$ does not admit a non-trivial elementary splitting, since $G$ does not split over elementary subgroups by assumption.

Putting everything together, we obtain the following result, which asserts that the collection of obstacles on the boundary is indeed ``well-behaved enough". This is particularly useful, as it guarantees that we can use all the results from Section \ref{quasiarcssection} to prove Theorems \ref{maincorollary}, \ref{maintheorem} and Corollary \ref{generalnonplanar}.
\begin{corollary}\label{Egconditions}
     Let $G$ be a one-ended hyperbolic group such that $\bdry G$ is non-planar, and let $g\in G$ be an infinite order element. Let $\V$ be the collection consisting of all limit sets of left cosets of $E(g)$ in $\bdry G$. Then, $\V$ is $L$-separated, where $L$ denotes the $L$-linear connectedness constant of the Gromov boundary $\bdry G$. 
     
     Moreover, each compact subset $V\in \V$ is $L$-porous and $L$-avoidable on scales below $D(V)$, with scale function $D(V)\colon \mathcal{V}\to \bdry G$ defined by $D(V)=e^{-\varepsilon d_V}$ for all $V \in \mathcal{V}$. Here, $d_V$ denotes the distance between the basepoint $w$ of the geodesic rays that define $\bdry G$ and the left coset of $E(g)$ whose limit points make up $V\in \V$. 
\end{corollary}
\begin{proof}
   First, every subset of $\V$ is compact. To see this, recall that a subset in $ \V$ consists of two points, namely $\{hg^{\pm\infty}\}$ for $h\in G$. Proposition \ref{lem:mssep} ensures that there is a constant for which the limit sets of left cosets of $E(g)$ are $L_1$-separated in the boundary of $G$. In addition, $G$ is not virtually Fuchsian, since by assumption $\bdry G$ is non-planar; it is a standard result that virtually Fuchsian groups have planar boundary. Using Theorem \ref{gm7.9}, $G$ is locally connected with no (local or global) cut points. Hence, we can again use Proposition \ref{lem:mssep} to obtain that there exist constants $L_2,L_3$ such that every $V\in \V$ is $L_2$-avoidable and $L_3$-porous on scales below $D(V)$. We require the collection to be $L$-separated, $L$-porous and $L$-avoidable, where $L$ is the linear connectedness constant. To ensure this, we can safely assume $L \geq \mathrm{max}\{L_1,L_2,L_3\}$, since increasing $L$ does not affect the $L$-linear connectedness condition.
\end{proof}
\section{Mapping graphs into the quotient boundary}\label{bdrymapsection}
In the previous two sections, we defined a collection of obstacles on the boundary and established a notion of avoiding them. Namely, given a hyperbolic group $G$ and an infinite order element $g$, the obstacles are defined as the points on the boundary corresponding to the left cosets of $E(g)\leq G$, where $E(g)$ is the maximal elementary subgroup containing $g$. This was done to ensure that graphs on the boundary of $G$ can survive the quotient and appear in the boundary of the quotient group. In this section, we prove that avoiding the above set of obstacles is indeed enough to appear as a graph in the quotient boundary. 

Following the notation set in Definition \ref{obstacles}, let $G$ be a hyperbolic, one ended group, and $g\in G$ an infinite order element. We will use the notation $\gbar$ for $G\big/ \langle \langle g^n \rangle\rangle$ for some $n$ and $\pi\colon G\to \gbar$ for the quotient map. Recall that $(G,E(g))$ is a relatively hyperbolic group by Theorem \ref{elementarysubgroup}. 
Let $\mathcal{V}$ denote the representatives on the boundary of left cosets of $E(g)$, namely sets of the form $\{hg^{\pm\infty}\}$ for $h\in G$. 

Next, we give a precise definition of the subset of the boundary that consists of the points that avoid obstacles. For some fixed $\lambda$, let $\bdry_{\lambda}G$ denote the subset of $\bdry G$ that consists of points that $\lambda$-avoid $V$, for all $V\in \V$ (Definition \ref{avoids}). By definition of the Gromov boundary, these points correspond to a collection of geodesic rays in $G$, based at some fixed basepoint $w$.

In general, there is no map between $\bdry G$ and $\bdry \gbar$ induced by the quotient. However, we will show that there exists a topological embedding from $\bdry_{\la}G$ to $\bdry \gbar$. As we will see in Section \ref{sec:6}, given a finite simplicial metric graph $\G$ of valence $\leq 4$ in $\bdry G$, we can construct a new topological embedding of (a graph with minor) $\G$ in $\bdry_{\la}G$ and thus, using the results in this section, in $\bdry \gbar$.
\subsection{Transversal maps}
We first describe the notion of a \textit{transversal} quasi-isometric embedding; this will ensure that the rays in $G$ corresponding to $\bdry_{\lambda}G$ ``have bounded interactions'' with left cosets of $E(g)$. 

\begin{definition}\cite[Definition 2.8]{mackaysisto}
    Let $(G,\Psub)$ be a relatively hyperbolic group, and $Y$ a metric space. Given some function $\eta\colon[0,\infty)\to[0,\infty)$, we say that a quasi-isometric embedding $f\colon Y\to G$ is an $\eta$-transversal if for all $M\geq 0$, $g\in G$, $P\in \Psub$, $$\diam (f(Y)\cap N(gP,M)\leq \eta(M).$$
\end{definition}
Since $G$ is hyperbolic relative to $E(g)$, transversal geodesics in this setting correspond to exactly the geodesics with bounded interactions with the left cosets of $E(g)$. The following lemma is implicit in the proof of \cite[Proposition 3.5]{mackaysisto}, see, in particular, Equation (3.6) therein.
\begin{lemma}\label{transversal}
    For all $\lambda$, there exists $\eta\colon [0,\infty)\to[0,\infty)$ such that all $p\in \bdry_\lambda G$ are represented by an $\eta$-transversal geodesic ray in $G$ starting at $w$.
\end{lemma}
The following special case of \cite[Proposition 2.12]{mackaysisto} shows that the image of an $\eta$-transversal geodesic ray in $G$ will be a quasi-geodesic ray in $\gbar$, which we will use to define the boundary map. Note, the constant for the quasi-isometric embedding in \cite[Proposition 2.12]{mackaysisto} is not quantified in the statement of the result, but careful tracking of the constant dependencies throughout the proof supply the quantifiers written below.
\begin{proposition}\label{qgeodesictransversal}
For all hyperbolic groups $G$, all infinite order $g\in G$, for all transversality functions $\eta\colon [0,\infty)\to [0,\infty)$, there exist $K, N_0>0$ such that for all multiples $n$ of $N_0$, for $\gbardef$ and $\gamma$ an $\eta$-transversal geodesic ray (resp. line) in $G$, we have that $\pi \circ \gamma$ is a $K$-quasi geodesic (resp. line) in $\gbar$.
\end{proposition}
\subsection{The boundary map} Before introducing the boundary map, we first fix some notation for the boundary. Fix a basepoint $w\in G$. Let $\gamma_1,\gamma_2 \colon [0,\infty)\to G$ be two geodesic rays in $G$ starting at $w$, and $t$ be such that $\gamma_1(t)$ and $\gamma_2(t)$ are the last points along the geodesic rays that satisfy $d_G(\gamma_1(t),\gamma_2(t))\leq 2\delta$. Given the setup of the previous sentence, the \textit{Gromov product} $(\gamma_1|\gamma_2)_w=t$. When working in $\bdry \gbar$, we define the Gromov product in an analogous way and use the same notation. The Gromov boundary is equipped with a \textit{visual metric} $\rho$, that satisfies $\rho(\gamma_1,\gamma_2) \asymp \mathrm{exp}(-\varepsilon (\gamma_1|\gamma_2)_w)$. To be more precise, the visual metric is bi-Lipshitz to $\mathrm{exp}(-\varepsilon (\gamma_1|\gamma_2)_w)$. We will write $\rho_{G}$ and $\rho_{\gbar}$ when referring to the visual metric in $\bdry G$ and $\bdry \gbar$, respectively. 

We can now introduce the map from $\bdry_\lambda G$ to $\bdry \gbar$. Let $\varphi\colon \bdry_\lambda G\to \bdry \gbar$ be the map defined as follows: letting $p\in \bdry_\lambda G$, let $\gamma_p$ be a geodesic ray in $G$ representing $p$. By Proposition \ref{qgeodesictransversal}, $\pi\circ \gamma_p$ is a quasi-geodesic ray in $\gbar$. Let $\varphi(p)$ be a geodesic in $\gbar$ which is bounded Hausdorff distance from $\pi \circ \gamma_p$. Note, since $G$ is hyperbolic, such a geodesic exists by stability of quasi-geodesics. 

\begin{proposition}\label{prop:bdrymap}
    Let $\lambda >0$. The map $\varphi\colon \bdry_\lambda G \to \bdry \gbar$ is well-defined, injective and continuous. 
\end{proposition}
\begin{proof}
Fix $\lambda >0$. Let $\delta$ denote the hyperbolicity constant of $G$, and we increase it so that ideal triangles (some of whose vertices can be on $\bdry G$) are also $\delta$-slim. Let $p,q\in \bdry_\lambda G$, and $\gamma_p,\gamma_q$ be geodesic rays in $G$ representing $p,q$, respectively. 
\\\\
\textit{$\varphi$ is well-defined.} Assume $p=q$. Then, the geodesic rays $\gamma_p,\gamma_q$ are bounded Hausdorff distance from each other. Since the quotient map $\pi$ is 1-Lipschitz, the same is true for $\pi\circ \gamma_p, \pi\circ\gamma_q$. By construction of $\varphi$, the representatives of their images $\varphi(p),\varphi(q)$ will also be bounded Hausdorff distance from each other, hence $\varphi(p)=\varphi(q)\in \bdry \gbar$.
\\\\
\textit{Continuity.} We will show that the map $\varphi$ is H{\"o}lder continuous; there exist $A,\alpha$ such that $\rho_{\gbar}(\varphi(p),\varphi(q))\leq A\rho_{G}(p,q)^\alpha$ for all $p,q\in \bdry_\lambda G$.

Let $(\gamma_p|\gamma_q)_w=t$, and $x=\gamma_p(t), y=\gamma_q(t)$. Since $\pi$ is 1-Lipschitz, $d_{\gbar} (\pi(x),\pi(y))\leq 2\delta$. In addition, since $\pi\circ \gamma_p,\pi\circ \gamma_q$ are $K$-quasi-geodesics, $d_{\gbar}(w,x)=d_{\gbar}(w,y)\geq (\gamma_p|\gamma_q)_w/K-K$. The quasi-geodesics $\pi\circ \gamma_p,\pi\circ \gamma_q$ are of uniform bounded distance $R$ from the geodesics $\varphi(p),\varphi(q)$. Hence, for two points $x'= \varphi(p)(t'), y'= \varphi(q)(t')$, we have $d_{\gbar}(x',y')\leq 2R+2\delta$, and $t'\geq (\gamma_p|\gamma_q)_w/K-K-R$. 

The Gromov product has the following standard property: if two points along $\varphi(p),\varphi(q)$ are bounded distance from each other, then there exists $D'$ such that $(\varphi(p)|\varphi(q))_{\pi(w)}\geq D-D'$, where $D$ is the distance of those points from the basepoint $\pi(w)$. Using this, we obtain that there exists $D'$ such that $(\varphi(p)|\varphi(q))_{\pi(w)}\geq t/K-K-R-D'$. Therefore, taking $\rho_{\gbar}(\varphi(p),\varphi(q))\asymp \mathrm{exp}(-\varepsilon (\varphi(p)|\varphi(q))_{\pi(w)})$, it follows that: $$\rho_{\gbar}(\varphi(p),\varphi(q))\leq \mathrm{exp}(-\e(t/K-K-R-D))=A\,\mathrm{exp}(-\varepsilon(\gamma_p|\gamma_q)_w)^\alpha \asymp A\rho_G(p,q)^\alpha$$ where $A= \mathrm{exp}(K+R+D)$ and $\alpha =1/K$.
\\\\
\textit{Injectivity.} Let $p,q\in \bdry_\lambda G$ be distinct points. For any two such points on the boundary, there exists a bi-infinite geodesic ray $\gamma_{pq}\colon(-\infty,\infty)\to G$ connecting them, such that $\gamma_{pq}|_{[0,\infty)}$ fellow-travels with $\gamma_p$, and $\gamma_{pq}|_{(-\infty,0]}$ with $\gamma_q$.  
To show injectivity, it suffices to show that $\gamma_{pq}$ is an $\eta'$-transversal for some $\eta'$, since, by Proposition \ref{qgeodesictransversal}, the transversals $\gamma_p,\gamma_q,\gamma_{pq}$ will then be mapped to quasi-geodesics in $\gbar$ and hence $\varphi(p)\neq \varphi(q)$.

 Now we show that there exists $\eta'$ such that $\gamma_{p,q}$ is an $\eta'$-transversal:  let $g\in G$ be an infinite order element, and recall that $(G,E(g))$ is relatively hyperbolic. We will show that there exists a function $\eta'$ such that for all $M\geq 0, h\in G$, $\diam (\gamma_{pq}\cap N(hE(g),M))\leq \eta'(M)$.
 
 Consider the semi-ideal geodesic triangle $\gamma_p\cup\gamma_{pq}\cup \gamma_q$. By our choice of $\delta$, we have $\gamma_{p,q}\subset N(\gamma_p\cup\gamma_q,\delta)$. Now, assume $[x,y]\subset \gamma_{pq}$ is contained in the neighbourhood $N(hE(g),M)$ for some $M>0$. Let $x',y' \in \gamma_p\cup \gamma_q$ be $\delta$-close to $x,y$, respectively. We will assume that $x',y'$ lie on different geodesic rays $\gamma_p,\gamma_q$ respectively. Let $m$ be the midpoint of $[x,y]$ and assume, without loss of generality, that it is $\delta$-close to some point $m'\in\gamma_p$. There exists some constant $B$ such that any geodesic with endpoints $x',m'$ will be contained in $N(hE(g), B(M+\delta))$. This follows from the fact that $G$ is hyperbolic relative to $E(g)$, and hence the collection of left cosets of $E(g)$ is quasi-convex \cite[Theorem 4.1, Lemma 4.3, Theorem 8.5]{drutusapir}. Since $\gamma_p\in G$ is an $\eta$-transversal geodesic ray, we have that the length of the geodesic satisfies $[x',m']\leq \eta (\delta B +MB)$. Putting this all together, $$[x,y]=2[x,m]\leq d_G(x,x')+[x',m']+d_G(m',m)\leq 2\delta + \eta (\delta B +MB)$$ for any $[x,y]\subset \gamma_{pq}$ that is contained in $N(hE(g),M)$. An analogous argument holds if $x,y$ are both close to the same geodesic ray; to see this, note that the argument for using the midpoint $m$ can be avoided, and the rest of the argument is similar, taking $y,y'$ in place of $m,m'$. Hence, setting $\eta'(M)=2\delta + \eta (\delta B +MB)$, we get that for all $h\in G, M\geq 0$, $\diam (\gamma_{pq}\cap N(hE(g),M)\leq \eta'(M)$.  
\end{proof}
The above proposition is exactly what is required to map graphs in $\bdry G$ to graphs in $\bdry \gbar$. It guarantees that, as long as the points on $\bdry G$ that make up the graph are contained in $\bdry_\lambda G$ for some $\lambda$, the map $\varphi$ will map them continuously and injectively to points in $\bdry \gbar$ that make up a graph in the quotient boundary.
\section{Graphs}\label{sec:6}
In order to prove Theorem \ref{maintheorem} and, in turn, Theorem \ref{maincorollary} and Corollary \ref{generalnonplanar}, we need to show that when $K_{3,3}$ or $K_5$ (Figure \ref{fig:K33}) embeds in $\bdry G$, they can be modified to avoid a collection of obstacles and contain the original graph as an embedded minor. In fact, we will prove a more general result: we can do this for any finite simplicial graph of valence at most 4.

Given a graph in a metric space $Z$ and collection of obstacles $\V$ satisfying Assumption \ref{assumptiongeneral} as defined in Section \ref{quasiarcssection}, Lemma \ref{endpoints} can be applied to vertices to obtain new endpoints that $\alpha$-avoid $\V$ for some $\alpha$, and, for all $\varepsilon>0$, each edge between them will yield a new arc that is contained in its $\e$-neighbourhood. Next, applying Lemma \ref{quasiarcs} will yield a $\lambda$-quasi arc for each edge, which $\lambda$-avoids $\V$ for some $\lambda$. The following shows that we have some control over the output of this procedure:
\begin{lemma}\label{lem:creatine}
    Let $(Z,\rho)$ be an $N$-doubling, $L$-linearly connected, compact metric space. Let $\V$ be an $L$-separated collection of compact subsets $V\subseteq Z$ and $D\colon  \V \to (0, \infty)$ be a scale function, such that $V \in \V$ is $L$-porous and $L$-avoidable on scales below $D(V)$ (Assumption \ref{assumptiongeneral}). Let $\G$ be a finite metric simplicial graph of valence at most 4, and $f\colon \G \to Z$ a topological embedding. 
    Then there exists $\lambda$ and a continuous map $\hat{f}\colon \G\to Z$, so that $\hat{f}(E)$ is a $\lambda$-quasi arc that $\lambda$-avoids $\V$ for all edges $E\in E(\G)$. Moreover, for all $E,E'\in E(\G)$ with $E\neq E'$:
    \begin{enumerate}[(i)]
        \item  if $E,E'$ do not share any vertices, then $\hat{f}(E)\cap \hat{f}(E')=\emptyset$ 
        \item if $E,E'$ share a vertex, then $\hat{f}(E)\cap \hat{f}(E')\subseteq N_{d/C}(E\cap E')$, where $d$ is half the minimum distance between distinct vertices in $V(\G)$, and $C=2\lambda+1$.
    \end{enumerate}
\end{lemma}
\begin{proof}
Let $\varepsilon>0$ be sufficiently small as determined later in the proof. Let $V,W\in V(\G)$ be the images of vertices under $f$. We denote the image of the edge that has endpoints $V$ and $W$ by $E_{V,W}$. We obtain the required map as follows: 

First, apply Lemma \ref{endpoints} to $E_{V,W}$. Let $E'_{v,w}$ denote the arc connecting the two new endpoints, $v,w$. $E'_{v,w}$ is contained within the $\e$-neighbourhood of $E_{V,W}$. Now, apply Lemma \ref{quasiarcs} to $E'_{v,w}$. Let $e_{v,w}$ denote the resulting $\lambda$-quasi arc which $\e$-follows $E'_{v,w}$ and $\lambda$-avoids $\V$. The arc $e_{v,w}$ is contained in the $2\e$-neighbourhood of the original edge. We define $\hat{f}$ to be the map which takes $E_{V,W}$ to $e_{v,w}$ for each edge of $\G$.

\textit{Continuity}: Since $\hat{f}$ is continuous when restricted to each edge, and $\G$ is locally finite, by the gluing lemma $\hat{f}\colon \G\to Z$ is continuous. 

\textit{Disjoint edges remain disjoint}: Let $d_{\mathrm{min}}$ be the minimum distance in $Z$ between the images of two distinct, disjoint edges of $\G$. Let $\e > d_{\mathrm{min}}/8$. Now, let $E_{V,W}$ and $E_{U,T}$ be images of a pair of disjoint edges of $\G$. By Lemma \ref{endpoints}, the resulting arcs $E'_{v,w}$,$E'_{u,t}$ are contained within the $\varepsilon$-neighbourhood of $E_{V,W},E_{U,T}$, respectively. Since $\e> d_{\mathrm{min}}/8$, the arcs $E'_{v,w}$ and $E'_{u,t}$ remain disjoint. Applying Lemma \ref{quasiarcs}, we obtain two quasi-arcs which $\e$-follow $E'_{v,w}$ and $E'_{u,t}$, respectively. In particular they are within distance $\e$ from them. These arcs remain disjoint. To see this, note that we have picked $\e$ to be less than half the distance between $E'_{v,w}$ and $E'_{u,t}$. 

\textit{Intersections remain within bounded neighbourhoods}: Now assume $E_{V,W}$ and $E_{V,U}$ share a vertex $V$. We follow an approach that is similar to the one for disjoint edges. We instead pick the minimum distance between $E_{V,W}$ and $E_{V,U}$ outside of the $2\e$-neighbourhood around $V$. Let $d'_{\mathrm{min}}$ be the minimum such distance across the images of all pairs of edges of $\G$ that share endpoints. By the above argument, picking $\e<d'_{\mathrm{min}}/8$ guarantees that the quasi-arcs $e_{v,w}$ and $e_{v,u}$ are disjoint outside of the $2\e$-neighbourhood of $V$.

Putting this all together, we can choose $\e < \mathrm{min}\{d/C, d_{\mathrm{min}}/8, d'_{\mathrm{min}}/8\}$, which implies that all conditions above are satisfied.
\end{proof}
We now have a way to modify the image of $\G$ to avoid obstacles, and can proceed to construct a graph contained in the modified version of it (that is, $\hat{f}(\G)$) which contains $\G$ as a minor, hence proving a significant part of the main theorem. An obstruction to this is that it is possible $\hat{f}(\G)$ is not embedded into $Z$, since, within the $2\e$-neighbourhoods around the vertices, the procedure above may have introduced a number of new intersections. Hence, we must first ``clean up" $\hat{f}(\G)$ by minimising intersections around the vertices to find a minor that topologically embeds into $\hat{f}(\G)$ and thus $Z$. it is worth noting that, in the case that the graph is of valence at most 3, the modified graph is in fact homeomorphic to $\G$; we prove this as a remark within the proof of the below proposition. 
\begin{proposition}\label{tren}
    Let $(Z,\rho)$ be a metric space. Let $\G$ be a finite metric simplicial graph of valence at most $4$, $E(\G)$ be its set of edges, and $V(\G)$ its set of vertices. Let $f\colon \G \to Z$ be a continuous function such that $f(E)$ is a $\lambda$-quasi-arc for all edges $E\in E(\G)$. Assume the following holds for all $E,E'\in E(\G)$ with $E\neq E'$:
    \begin{enumerate}[(i)]
        \item \label{assumption1} if $E,E'$ do not share any vertices, then $f(E)\cap f(E')=\emptyset$ 
        \item\label{assumption} if $E,E'$ share a vertex, then $f(E)\cap f(E')\subseteq N_{d/C}(E\cap E')$, where $d$ is half the minimum distance between distinct vertices in $V(\G)$, and $C=2\lambda+1$ 
    \end{enumerate}
    then there exists $\G'\subseteq f(\G)$ such that $\G'$ contains $\G$ as a minor.
\end{proposition}
\begin{proof}
   We first set some notation. We use $e\cdot e'$ to denote the concatenation of two arcs $e,e'$. We will denote edges of $\Gamma$ by $E_{V,W}$, where $V,W\in V(\G)$ are the vertices that are incident to this edge. 
   We consider the image of $\G$ under $f$: \begin{itemize}
        \item for each vertex $V\in V(\G)$ we denote the point $f(V)$ by $v$.
        \item for each edge $E_{V,W}$ we have a $\lambda$-quasi-arc $f(E_{V,W})$ with endpoints $v=f(V)$ and $w=f(W)$. We let this quasi-arc be denoted by $e_{v,w}$.
    \end{itemize} 
    Choose points $x_{v,w}$ on the arcs $e_{v,w}$ as follows: \\ Let $x_{v,w}$ be the last point along the arc $e_{v,w}$, starting at $v$, that realises the distance $\rho(v,x_{v,w})=2d/C$.
    \\\\
    \textbf{\textit{Subarcs avoid neighbourhoods around the vertices.}} Consider the subarc of $e_{v,w}$ with endpoints $x_{v,w},x_{w,v}$, denoted $\hat{e}_{v,w}=e_{v,w}[x_{v,w},x_{w,v}]$. We claim that $\hat{e}_{v,w}\cap N_{d/C}(v\cup w)=\emptyset$. For a contradiction, assume that $\hat{e}_{v,w}\cap N_{d/C}(v\cup w)\neq \emptyset$, as shown in Figure \ref{fig:diameter}, and assume that it intersects the $d/C-$neighbourhood around $w$ (if it only intersects the neighbourhood around $v$, an analogous argument holds). 
    \begin{figure}
    \centering
\begin{tikzpicture}
    \filldraw[myorange] (0,0) circle (17pt);
    \filldraw[myorange] (5,0) circle (17pt);
    \node[myorange] at (-0.6,0.8) {$N_{d/C}(v)$};
    \node[myorange] at (5.6,0.8) {$N_{d/C}(w)$};
    \filldraw[black] (0,0) circle (1pt);
    \node[black] at (0,0.2) {$v$};
    \filldraw[black] (5,0) circle (1pt);
    \node[black] at (5,0.2) {$w$};
    \filldraw[black] (1,-0.63) circle (1pt);
    \node[black] at (1,-0.4) {$x_{v,w}$};
    \filldraw[black] (3.93,0.2) circle (1pt);
    \node[black] at (3.8,0.45) {$x_{w,v}$};
    \draw (0,0) .. controls (1.6,1) .. (4.5,0);
    \draw (4.5,0) .. controls (5,-0.2) .. (0.4,-0.2);
    \draw (0.4,-0.2) .. controls (0,-0.2) .. (0.4,-0.4);
    \draw (0.4,-0.4) .. controls (2.5,-1.2) .. (4.6,-0.4);
    \draw (4.6,-0.4) .. controls (5,-0.2) .. (5,0);
\end{tikzpicture}
\caption{A case where the subarc with endpoints $x_{v,w}$ and $x_{w,v}$ intersects $N_{d/C}{(v\cup w)}$.}
\label{fig:diameter}
\end{figure}
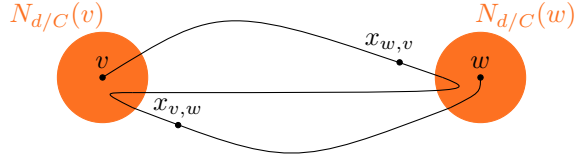
    By definition of the point $x_{w,v}$, the subarc $\hat{e}_{v,w}$ is disjoint from $N_{d/C}(w)$. This implies that $x_{v,w}\in e_{v,w}[x_{w,v},w]$, so, in particular $e_{v,w}[v,x_{v,w}]\cap N_{d/C}(w)\neq \emptyset$. Thus by definition of $d$ and $C$,\begin{equation}\label{subarcs}
        \mathrm{diam}(e_{v,w}[v,x_{v,w}])>\rho(v,w)-d/C\geq d-d/C=\frac{(C-1)}{C}d = \frac{2\lambda d}{2\lambda+1}. 
    \end{equation} 
    However, by our assumption, $e_{v,w}$ is a $\lambda$-quasi arc, therefore $$
        \mathrm{diam}(e_{v,w}[v,x_{v,w}])\leq \lambda \rho(v,x_{v,w}) = 2\lambda d/C = \frac{2\lambda d}{2\lambda+1}$$
    which contradicts (\ref{subarcs}). Hence the subarc of $e_{v,w}$ with endpoints $x_{v,w},x_{w,v}$ is disjoint from $N_{d/C}(v\cup w)$. 
\\\\
\textit{\textbf{Minimising intersections around the vertices.}} Now, when constructing $\G'\subseteq f(\G)$, we need to ensure pairs of edges of $\G'$ intersect only at their vertices. \\\\
\textit{Vertices of valence 3.} First, we consider the case that the vertex $v$ has valence 3, so $v$ is shared by 3 edges. Hence, under $f$, there are three $\lambda$-quasi arcs with endpoint $v$, say $e_{v,w},e_{v,u},e_{v,t}$. By Assumption (\ref{assumption}), any pairwise intersections between $e_{v,w}, e_{v,u}, e_{v,t}$ are contained within $N_{d/C}(v)$. Let $p\in e_{v,w}[v,x_{v,w}]$ be the last point along $e_{v,w}$, starting at $v$, such that $p$ intersects $ e_{v,u}[v,x_{v,u}]\cup e_{v,t}[v,x_{v,t}]$. Without loss of generality, assume $p\in e_{v,u}[v,x_{v,u}]$. Now, let $q\in e_{v,t}[v,x_{v,t}]$ be the last point along $e_{v,t}[v,x_{v,t}]$, starting at $v$, which intersects $e_{v,u}[v,x_{v,u}]$. We have the following cases: 
    \begin{enumerate}
        \item $q\in e_{v,u}[v,p]$: we let $e_{v,w}'=e_{v,w}[p,x_{v,w}]$, $e_{v,u}'=e_{v,u}[p,x_{v,u}]$, and $e'_{v,t}=e_{v,t}[q,x_{v,t}]$. We denote the subarc of $e_{v,u}$ with endpoints $q$ and $p$ by $\tilde{e}_{q,p}$. This case is illustrated in Figure (\ref{fig:case2}). \label{case1}
        \item $q\in e_{v,u}[p,x_{v,u}]$: in this case, we let $e'_{v,w}=e_{v,w}[p,x_{v,w}]$, $e_{v,u}'=e_{v,u}[q,x_{v,u}]$, $e_{v,t}'=e_{v,t}[q,x_{v,t}]$, and let $\tilde{e}_{p,q}$ denote the subarc of $e_{v,u}$ with endpoints $p$ and $q$. This case is illustrated in Figure (\ref{fig:case3}). \label{case2}
    \end{enumerate}
    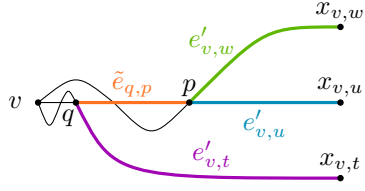
\begin{figure}
    \centering
\begin{tikzpicture}
    \draw[mygreen, very thick] (2,0) .. controls (3,1) .. (4,1);
    \draw[myblue, very thick] (2,0) -- (4,0);
    \filldraw[black] (4,1) circle (1pt);
    \node[black] at (4,1.2) {$x_{v,w}$};
    \filldraw[black] (4,0) circle (1pt);
    \node[black] at (4,0.2) {$x_{v,u}$};
    \filldraw[black] (0,0) circle (1pt);
    \node[black] at (-0.3,0) {$v$};
    \draw (0,0) .. controls (0.5,0.4) .. (1,0);
    \draw (1,0) .. controls (1.5,-0.5) .. (2,0);
    \draw (0,0) -- (0.5,0);
    \draw (0,0) .. controls (0.15,-0.4) .. (0.3,0);
    \draw (0.3,0) .. controls (0.4,0.2) .. (0.5,0);
    \draw[myorange, very thick] (2,0) -- (0.5,0);
    \draw[mypurple, very thick] (0.5,0) .. controls (1,-1) .. (4,-1);
    \filldraw[black] (2,0) circle (1pt);
    \node[black] at (2,0.2) {$p$};
    \filldraw[black] (0.5,0) circle (1pt);
    \node[black] at (0.4,-0.2) {$q$};
    \filldraw[black] (4,-1) circle (1pt);
    \node[black] at (4,-0.8) {$x_{v,t}$};
    \node[mygreen] at (2.3,0.8) {$e_{v,w}'$};
    \node[myblue] at (3,-0.3) {$e_{v,u}'$};
    \node[mypurple] at (2.3,-0.7) {$e_{v,t}'$};
    \node[myorange] at (1.25,0.25) {$\tilde{e}_{q,p}$};

\end{tikzpicture}
\caption{A possibility for case 1: the new vertices are $p$ and $q$.}
\label{fig:case2}
    \end{figure}
\begin{figure}
    \centering
\begin{tikzpicture}
    \draw[mygreen, very thick] (1.5,0) .. controls (2,1) .. (4,1);
    \draw[myblue, very thick] (3,0) -- (4,0);
    \filldraw[black] (4,1) circle (1pt);
    \node[black] at (4,1.2) {$x_{v,w}$};
    \filldraw[black] (4,0) circle (1pt);
    \node[black] at (4,0.2) {$x_{v,u}$};
    \filldraw[black] (0,0) circle (1pt);
    \node[black] at (-0.3,0) {$v$};
    \draw (0,0) .. controls (0.4,0.4) .. (0.75,0);
    \draw (0.75,0) .. controls (1.15,-0.5) .. (1.5,0);
    \draw (0,0) -- (1.5,0);
    \draw (0,0) .. controls (1,-1) .. (2,0);
    \draw (2,0) .. controls (2.5,0.5) .. (3,0);
    \draw[myorange, very thick] (1.5,0) -- (3,0);
    \node[myorange] at (2.25,-0.25) {$\tilde{e}_{p,q}$};
    \draw[mypurple, very thick] (3,0) .. controls (3.7,-1) .. (4,-1);
    \filldraw[black] (1.5,0) circle (1pt);
    \node[black] at (1.5,0.2) {$p$};
    \filldraw[black] (3,0) circle (1pt);
    \node[black] at (2.9,-0.2) {$q$};
    \filldraw[black] (4,-1) circle (1pt);
    \node[black] at (4,-0.8) {$x_{v,t}$};
    \node[mygreen] at (3,1.3) {$e_{v,w}'$};
    \node[myblue] at (3.3,0.3) {$e_{v,u}'$};
    \node[mypurple] at (3.2,-0.7) {$e_{v,t}'$};

\end{tikzpicture}
\caption{A possibility for case 2: the new vertices are $p$ and $q$.}
\label{fig:case3}
    \end{figure}
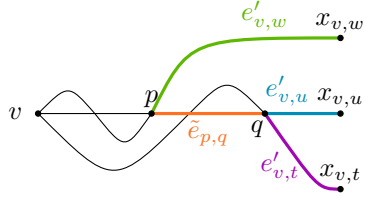
We claim that each of the four arcs constructed in both cases introduce no new intersections, and that these new arcs only share endpoints. We prove the first case, the second follows using similar arguments. Firstly, all four arcs are indeed arcs since they are subarcs of $e_{v,w},e_{v,u}$ and $e_{v,t}$. By definition of $p$ as the last point of intersection, $e'_{v,w}$ intersects the other arcs only at the endpoint $p$. Note that, in Case (\ref{case1}), $p$ is also the last point of intersection along $e_{v,u}$, hence the same holds for $e'_{v,u}$. An analogous argument holds for the point $q$ and $e'_{v,t}$ being disjoint from other arcs apart from at $q$. Finally, the subarc $\tilde{e}_{q,p}$ does not intersect any other arcs apart from at the endpoints, since by construction $e'_{v,w}, e'_{v,u}$ and $e'_{v,t}$ do not intersect $e_{v,u}[v,p]$ at any points other than $p$ and $q$. Moreover, none of the four arcs intersect $e_{v,w}[x_{v,w},w],e_{v,u}[x_{v,u},u]$ or $e_{v,t}[x_{v,t},t]$ (apart from at the shared endpoints $x_{v,w}$, $x_{v,u}$, $x_{v,t}$), because of assumptions (\ref{assumption1}), (\ref{assumption}).
\\\\
\textit{Vertices of valence 2.} In the case that the vertex $v$ has valence 2, an analogous argument holds: letting $e_{v,w},e_{w,v}$ be the images of edges adjacent to $v$, we pick the point $p\in e_{v,w}[v, x_{v,w}]$ be the last point along $e_{v,w}$, starting at $v$, that intersects $e_{w,v}$. By definition of $p$ there are no points of intersections on the subarcs $e'_{v,w}=e_{v,w}[p,x_{v,w}]$, $e_{v,u}'=e_{v,u}[p,x_{v,u}]$, other than $p$. 
\\\\
\textit{Vertices of valence 1.} Letting $e_{v,w}$ be the image of the edge adjacent to $v$, $e_{v,w}$ cannot self-intersect as, by assumption, it is a quasi-arc. Furthermore, $e_{v,w}[v,x_{v,w}]$ cannot intersect any other edges by Assumption (\ref{assumption1}). Hence we let $e'_{v,w}=e_{v,w}[v,x_{v,w}]$.
\\\\
\textit{Vertices of valence 4.} Now, let $v$ be a vertex of valence 4. We have left this case for last since we need to adjust the construction used for valence 3 to account for the extra edge. We do so as follows: let $e_{v,w},e_{v,u},e_{v,t}$ and $e_{v,y}$ be the arcs sharing the vertex $v$. We pick the point $p$ as above; let $p\in e_{v,w}[v,x_{v,w}]$ be the last point on this arc that intersects $e_{v,u}[v,x_{v,u}]\cup e_{v,t}[v,x_{v,t}]\cup e_{v,y}[v,x_{v,y}]$, assuming without loss of generality that $p\in e_{v,u}[v,x_{v,u}]$. Now, pick $q\in e_{v,u}[v,x_{v,u}]$ to be the last point of intersection between $e_{v,u}[v,x_{u,t}]$ and $e_{v,w}[p,x_{v,w}]\cup e_{v,t}[v,x_{v,t}]\cup e_{v,y}[v,x_{v,y}]$. Let $r\in e_{v,t}[v,x_{v,t}]$ be the last point of intersection with $e_{v,w}[p,x_{v,w}]\cup e_{v,u}[q,x_{v,u}]\cup e_{v,y}[v,x_{v,y}]$. Finally, pick $s\in e_{v,y}[v,x_{v,y}]$ to be the last point intersecting any of the other three arcs. 

We have the following two cases: 
\begin{enumerate}
    \item First, if $s\in e_{v,w}[p,x_{v,w}]\cup e_{v,u}[q,x_{v,u}]\cup e_{v,s}[r,x_{v,t}]$ then by construction it must be equal to one (or more) of the points $p,q,r$. An example of this is illustrated in Figure \ref{fig:valence42}.
    \item Otherwise, none of $e_{v,w}[p,x_{v,w}]$, $e_{v,u}[q,x_{v,u}]$, $e_{v,t}[r,x_{v,t}]$ intersect the arc $e_{v,y}[v,x_{v,y}]$, hence by construction at least two out of the points in $\{p,q,r\}$ must be equal. An example of this is illustrated in Figure \ref{fig:valence41}.
\end{enumerate}  In both cases, we have three points of intersection between the four arcs, included in the set $\{p,q,r,s\}$. Let $e'_{v,w}= e_{v,w}[p,x_{v,w}]$, $e'_{v,u}=e_{v,u}[q,x_{v,u}]$, $e'_{v,t}=e_{v,t}[r,x_{v,t}]$, $e'_{v,y}=e_{v,y}[s,x_{v,y}]$. In addition, whenever two of the points in the set $\{p,q,r,s\}$ are consecutive points on the same arc $e_{v,i}[v,x_{v,i}]$ for some $i\in \{w,u,t,y\}$, let the subarc of $e_{v,i}[v,x_{v,i}]$ which connects them be denoted by $\tilde{e}_{a,b}$, where $a\neq b$ are the points it connects, noting that these subarcs do not intersect any of the other arcs, apart from at the endpoints $a$ and $b$, by our construction. 
\begin{figure}
    \centering
     \begin{tikzpicture}
        \draw[very thick, myblue] (4,0) -- (5,0);
        \draw (0,0) -- (4,0);
        \draw (0,0) .. controls (1,-1) .. (2,0);
        \draw[very thick, mygreen] (2,0) .. controls (3,1) .. (5,1);
        \draw (0,0) .. controls (1.25,-0.5) .. (2.5,0);
        \draw (2.5,0) .. controls (3.75,0.7) .. (4,0);
        \draw[very thick, myorange] (4.25,-1) .. controls (4.5,-2) .. (5,-2);
        \draw (4,0) -- (4.25,-1); 
        \draw[very thick, mypurple] (4.25,-1) .. controls (4.25,-1) .. (5,-1);
        \draw (0,0) .. controls (1,-1) .. (4.25,-1);
        \filldraw[black] (2,0) circle (1pt);
        \filldraw[black] (4,0) circle (1pt);
        \filldraw[black] (4.25,-1) circle (1pt);
        \filldraw[black] (5,0) circle (1pt);
        \filldraw[black] (5,1) circle (1pt);
        \filldraw[black] (5,-1) circle (1pt);
        \filldraw[black] (5,-2) circle (1pt);
        \filldraw[black] (0,0) circle (1pt);
        \node[black] at (0,0.25) {$v$};
        \node[black] at (2,0.25) {$p$};
        \node[black] at (3.9,-0.25) {$q$};
        \node[black] at (3.8,-1.25) {$r=s$};
        \node[black] at (5,1.25) {$x_{v,w}$};
        \node[black] at (5,0.25) {$x_{v,u}$};
        \node[black] at (5,-0.75) {$x_{v,t}$};
        \node[black] at (5,-1.75) {$x_{v,y}$};
    \end{tikzpicture}
    \qquad
    \begin{tikzpicture}
        \draw[very thick, myblue] (4,0) -- (5,0);
        \draw (2,0) -- (4,0);
        \draw[very thick, mygreen] (2,0) .. controls (3,1) .. (5,1);
        \draw[very thick, myorange] (4.25,-1) .. controls (4.5,-2) .. (5,-2);
        \draw (4,0) -- (4.25,-1); 
        \draw[very thick, mypurple] (4.25,-1) .. controls (4.25,-1) .. (5,-1);
        \filldraw[black] (2,0) circle (1pt);
        \filldraw[black] (4,0) circle (1pt);
        \filldraw[black] (4.25,-1) circle (1pt);
        \filldraw[black] (5,0) circle (1pt);
        \filldraw[black] (5,1) circle (1pt);
        \filldraw[black] (5,-1) circle (1pt);
        \filldraw[black] (5,-2) circle (1pt);
        \node[black] at (2,0.25) {$p$};
        \node[black] at (3.9,-0.25) {$q$};
        \node[black] at (3.75,-1) {$r=s$};
        \node[black] at (5.5,1) {$x_{v,w}$};
        \node[black] at (5.5,0) {$x_{v,u}$};
        \node[black] at (5.5,-1) {$x_{v,t}$};
        \node[black] at (5.5,-2) {$x_{v,y}$};
        \node[mygreen] at (3.5,1.25) {$e'_{v,w}$};
        \node[myblue] at (4.5,0.25) {$e'_{v,u}$};
        \node[mypurple] at (4.6,-0.75) {$e'_{v,t}$};
        \node[myorange] at (4,-1.55) {$e'_{v,y}$};
        \node[black] at (3,0.25) {$\tilde{e}_{p,q}$};
        \draw[-stealth] (3,-0.6) -- (4,-0.6);
        \node[black] at (2.7,-0.55) {$\tilde{e}_{q,s}$};
    \end{tikzpicture}
    \caption{Cleaning up around a vertex of valence 4: a possibility for case 1.}
    \label{fig:valence42}
\end{figure}
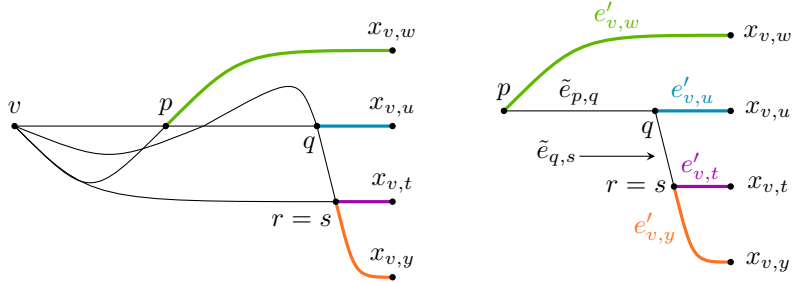
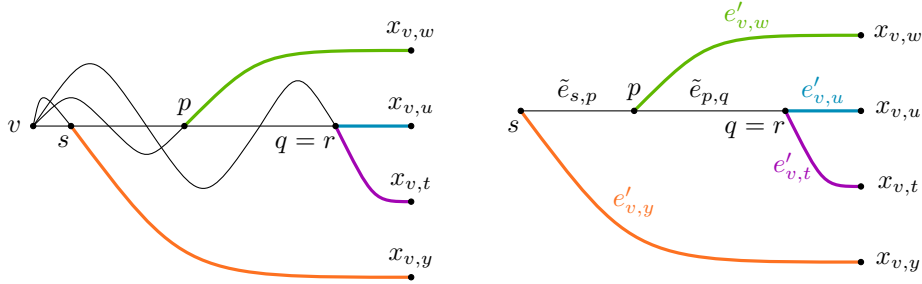
\begin{figure}
    \centering
    \begin{tikzpicture}
        \draw[very thick, myblue] (4,0) -- (5,0);
        \draw (0,0) -- (4,0);
        \draw[very thick, mygreen] (2,0) .. controls (3,1) .. (5,1);
        \draw[very thick, mypurple] (4,0) .. controls (4.5,-1) .. (5,-1);
        \draw[very thick, myorange] (0.5,0) .. controls (2,-2) .. (5,-2);
        \draw (0.5,0) .. controls (0.1,0.5) .. (0,0);
        \draw (4,0) .. controls (3.5,0.8) .. (3,0);
        \draw (3,0) .. controls (2.25,-1.1) .. (1.5,0);
        \draw (1.5,0) .. controls (0.75,1.1) .. (0,0);
        \draw (2,0) .. controls (1.5,-0.5) .. (1,0);
        \draw (1,0) .. controls (0.5,0.5) .. (0,0);
        \filldraw[black] (2,0) circle (1pt);
        \filldraw[black] (4,0) circle (1pt);
        \filldraw[black] (5,0) circle (1pt);
        \filldraw[black] (5,1) circle (1pt);
        \filldraw[black] (5,-1) circle (1pt);
        \filldraw[black] (5,-2) circle (1pt);
        \filldraw[black] (0,0) circle (1pt);
        \filldraw[black] (0.5,0) circle (1pt);
        \node[black] at (-0.25,0) {$v$};
        \node[black] at (0.4,-0.2) {$s$};
        \node[black] at (2,0.25) {$p$};
        \node[black] at (3.6,-0.25) {$q=r$};
        \node[black] at (5,1.25) {$x_{v,w}$};
        \node[black] at (5,0.25) {$x_{v,u}$};
        \node[black] at (5,-0.75) {$x_{v,t}$};
        \node[black] at (5,-1.75) {$x_{v,y}$};
\end{tikzpicture}
\qquad
\begin{tikzpicture}
        \draw[very thick, myblue] (4,0) -- (5,0);
        \draw (0.5,0) -- (4,0);
        \draw[very thick, mygreen] (2,0) .. controls (3,1) .. (5,1);
        \draw[very thick, mypurple] (4,0) .. controls (4.5,-1) .. (5,-1);
        \draw[very thick, myorange] (0.5,0) .. controls (2,-2) .. (5,-2);
        \filldraw[black] (2,0) circle (1pt);
        \filldraw[black] (4,0) circle (1pt);
        \filldraw[black] (5,0) circle (1pt);
        \filldraw[black] (5,1) circle (1pt);
        \filldraw[black] (5,-1) circle (1pt);
        \filldraw[black] (5,-2) circle (1pt);
        \filldraw[black] (0.5,0) circle (1pt);
        \node[black] at (0.4,-0.2) {$s$};
        \node[black] at (2,0.25) {$p$};
        \node[black] at (3.6,-0.25) {$q=r$};
        \node[black] at (5.5,1) {$x_{v,w}$};
        \node[black] at (5.5,0) {$x_{v,u}$};
        \node[black] at (5.5,-1) {$x_{v,t}$};
        \node[black] at (5.5,-2) {$x_{v,y}$};
        \node[mygreen] at (3.5,1.25) {$e'_{v,w}$};
        \node[myblue] at (4.5,0.25) {$e'_{v,u}$};
        \node[mypurple] at (4.1,-0.75) {$e'_{v,t}$};
        \node[myorange] at (2,-1.2) {$e'_{v,y}$};
        \node[black] at (1.25,0.25) {$\tilde{e}_{s,p}$};
        \node[black] at (3,0.25) {$\tilde{e}_{p,q}$};
       \end{tikzpicture}
    \caption{Cleaning up around a vertex of valence 4: a possibility for case 2.}
    \label{fig:valence41}
\end{figure}
\\\\
\textbf{\textit{Creating the new graph $\G'$.}} The final step will be to construct the edges and vertices of a graph $\G'$, which contains $\G$ as a minor and topologically embeds in $f(\G)$. For each arc $e_{v,w}\in f(E(\G))$ with endpoints $v,w\in f(V(\G))$, we take the following concatenation of subarcs: \begin{equation}\label{finaledges}
    e''_{v,w}=e_{v,w}'\cdot \hat{e}_{v,w}\cdot e'_{w,v}
    \end{equation}
    In the first part of the proof, we showed that $\hat{e}_{v,w}=e_{v,w}[x_{v,w},x_{w,v}]$ is indeed an arc, so the whole concatenation is an arc: notice that $e'_{v,w}$ and $e'_{w,v}$ are entirely disjoint from each other, and by the above argument they only intersect $\hat{e}_{v,w}$ once at their shared endpoint $x_{v,w}$ and $x_{w,v}$, respectively. Hence each of $e'_{v,w}, \hat{e}_{v,w}$ and $e'_{w,v}$ is an arc, and the only pairwise intersections between them are at the endpoints $x_{v,w},x_{w,v}$, so $e''_{v,w}$ is itself an arc.  
   Hence, we can embed the desired graph $\G'$ into $f(\G)$ by constructing its edges as follows:\begin{itemize}
       \item $e''_{v,w}$ as constructed above,
       \item any arc of the form $\tilde{e}_{a,b}$ included in the image of an edge of $\G$ adjacent to a vertex of valence 3 or 4, as constructed above.
   \end{itemize}
    To construct the set of vertices of $\G'$, for each vertex $v$ of $\G$, take the set $\{p,q,r,s\}$ defined above to be the set of vertices of $\G'$, noting that, whenever some points in $\{p,q,r,s\}$ are the same, they define the same vertex in $\G'$. In addition, note that, in this graph, no two edges intersect pairwise apart from at their shared endpoints. Indeed, by construction, the edges $e''_{v,w}$ consist of subarcs of the form $\hat{e}_{v,w}=e_{v,w}[x_{v,w},x_{w,v}]$, which are disjoint from any other arcs since they do not intersect the $d/C$-neighbourhood around the vertices, and subarcs of the form $e'_{v,w}$, which are also disjoint from other arcs by definition, as we chose their endpoint to be the last point along an edge that intersects any other arc. Finally, edges of the form $\tilde{e}_{a,b}$ are disjoint from other arcs since their endpoints are two consecutive points of intersection with other arcs. It follows that the arcs that define edges of $\G'$ only intersect at shared endpoints, and $\G'$ indeed defines an embedding of a graph in $Z$.
\\\\
\textbf{\textit{Minors of $\G'$.}} It remains to show that $\G$ is a minor of the graph $\G'$. It suffices to show that we can obtain $\G$ from $\G'$ by contracting edges; that is, identifying an edge of $\G'$ and its endpoints with a single point. Therefore, we consider the minor of $\G'$ obtained by contracting all edges of the form $\tilde{e}_{a,b}$, and denote it by $\G''$. Contracting such an edge identifies the points $a,b\in \{p,q,r,s\}$ to a single vertex. For each image of a vertex $v$ of $\G$, all endpoints in the set $\{p,q,r,s\}$ are identified to a single vertex, say $v'$, shared by exactly as many edges as $v$ was. Every edge of $\G''$ is in bijection with an edge of $\G$. To see this, take two vertices $v,w\in \G''$. We will show that these vertices are connected by an edge if and only if they were connected by an edge in $\G$. Indeed, if $v,w$ are distinct vertices connected by an edge, it will be of the form $e''_{v,w}$, hence the corresponding vertices $V,W\in V(\G)$ are connected by an edge $E_{V,W}$ of $\G$. Similarly, if $V,W$ are connected by an edge in $\G$, the corresponding vertices $v,w$ will be connected by an edge of the form $e''_{v,w}$ in $\G''$.
\end{proof}
\begin{remark}
    In the case that $\G$ has valence at most 3, the argument above can be adjusted slightly to show that we can get a resulting graph $\G'$ topologically embedded in $f(\G)$ that is isomorphic to $\G$. To see this, returning to Figures (\ref{fig:case2}) and (\ref{fig:case3}), notice that within each figure there exists a vertex of valence 3 with adjacent edges which do not intersect each other. More precisely, focusing on Figure (\ref{fig:case2}), let $\bar{e}_{v,t}$ be the concatenation of $\tilde{e}_{q,p}$ and $e'_{v,t}$. The endpoint $p$ is then the common vertex shared between the three edges $e'_{v,w},e'_{v,u}.\bar{e}_{v,t}$. The same argument can be made for the case in Figure (\ref{fig:case3}), we view the concatenation of $\tilde{e}_{p,q}$ and $e'_{v,w}$ as a single edge, denoted $\bar{e}_{v,w}$, letting $q$ be the shared vertex. The cases where a vertex is of valence less than 3 are straightforward. 

   We define the graph $\G'$ as before, but with the addition that we view edges of the form $\tilde{e}_{a,b}$ and $e'_{v,i}$, for some $a,b\in \{p,q,r,s\}$, $v,i\in f(V(\G))$, as a single edge $\bar{e}_{v,i}$. The resulting graph has the same number of vertices as $\G$, and, each edge adjacent to a vertex, say $E_{V,W}$, gives rise to an edge of the form either $e''_{v,w}$ or $\bar{e}_{v,w}$. Hence $\G'$ has the same number of edges and vertices, and the same valence as $\G$, and the two graphs are isomorphic. 

    Therefore, reducing the statement of the Proposition \ref{tren} to the case where $\G$ has valence at most 3, we get the stronger result that (a graph isomorphic to) $\G$ is contained within $f(\G)$.
\end{remark}
\section{Proving the main theorem}\label{sec:proofs}
\subsection{Main theorem}
The aim of this subsection is to discuss and prove the main theorem of the paper (Theorem \ref{maintheorem}). Before moving onto the proof of the theorem, we state a more general version of it which also holds, in which we allow the group $G$ to have torsion:
\begin{theorem}[Theorem \ref{maintheorem}]\label{maintheoremtorsion}
    Let $G$ be a one-ended hyperbolic group such that $G$ does not split over any elementary subgroup. Let $\G$ be a finite simplicial graph of valence at most 4, such that $\G$ embeds into $\bdry G$. Then for all infinite order elements $g\in G$, there exists $k>0$ such that for all multiples $n$ of $k$, we have that $\gbar=G\big/ \langle\langle g^n\rangle\rangle$ is hyperbolic and there exists a graph $\hat{\G}$ that topologically embeds into $\bdry \gbar$ and contains $\G$ as a minor.

    If, in addition, $G$ is torsion-free, the above holds for all large enough $n$.
\end{theorem}
\begin{proof} Let $G$ be a hyperbolic, one-ended group, such that $G$ does not split over the elementary subgroup $E(g)$ for any $g\in G$. By Remark \ref{dehnsetup}, there exists some $k$ such that $\langle g^k \rangle \unlhd E(g)$. Therefore, fixing $n$ to be any multiple of $k$, set $\gbar =G\big/ \langle\langle g^n\rangle\rangle$. Then, $\gbar$ can be realised as a Dehn filling of $G$.

We remark that, if, in addition, $G$ is torsion-free, by Remark \ref{dehnsetuptf}, $\gbar =G\big/ \langle\langle g^n\rangle\rangle$ is a Dehn filling of $G$ for all large enough $n$. Hence, the remainder of this proof holds for both the torsion as well as the torsion-free case for $n$ large enough.

We have that $\gbar$ is itself hyperbolic. Indeed, this is a result of \cite{delzant}, \cite{olshanskii} and \cite{gromov1987hyperbolic}. 

Now, let $(Z,\rho)$ denote $\bdry G$ with the usual visual metric on the boundary, described in Section \ref{bdrymapsection}. Assume $Z$ contains a graph of valence at most 4, and let $\G$ denote its image as a subset of $Z$. Moreover, let $\V$ denote the collection of the limit sets of left cosets of $E(g)$ in $Z$. We want to show that $\bdry\gbar$ contains a graph $\G'$ with $\G$ as a minor. We approach this by showing that there exists such a graph in $Z$ that avoids the points in $\V$. This suffices; by Lemma \ref{transversal} and Proposition \ref{qgeodesictransversal}, given a constant $\la$ such that a set of points in $Z$ $\la$-avoids $\V$, we can represent those points by geodesic rays in $G$ whose images under the quotient map are quasi-geodesic rays in $\gbar$. The corresponding map from $Z$ to $\bdry \gbar$ will then allow a graph that $\la$-avoids $\V$ to be mapped to a graph in $\bdry \gbar$ by Proposition \ref{prop:bdrymap}.

It was shown in Corollary \ref{Egconditions} that the collection of obstacles $\V$ on the boundary is $L$-separated, porous, and avoidable, where $L$ is the linear connectedness constant for $Z$. Therefore, applying Lemma \ref{lem:creatine} to $\G$, we get that there exists $\lambda$ and a continuous map $\hat{f}\colon \G\to Z$, so that $\hat{f}(E)$ is a $\lambda$-quasi arc that $\lambda$-avoids $\V$ for all edges $E\in E(\G)$. Moreover, the map $\hat{f}$ satisfies Conditions (\ref{assumption1}) and (\ref{assumption}) of Lemma \ref{lem:creatine}. Therefore, we can apply Proposition \ref{tren} to obtain a graph $\G'$ topologically embedded in $Z$ such that $\G$ is a minor of $\G'$. Since $\hat{f}(E)$ $\la$-avoids $\V$ for all edges $E$ of $\G$, we have that $\hat{f}(\G)$ $\la$-avoids $\V$ hence all edges of $\G'$ avoid $\V$. It follows that $\bdry \gbar$ contains a graph $\G'$ such that $\G'$ has minor $\G$, where $\G$ is a graph of valence at most 4 contained in $\bdry G$. 
\end{proof}
\subsection{Menger curve boundaries}
This subsection proves Theorem \ref{maincorollary}. Namely, we show that if $G$ is a hyperbolic, torsion-free group such that $\bdry G$ is a Menger curve, and $g\in G$ an infinite order element, then for all large enough $n$, $\bdry \gbar$ is also a Menger curve for $\gbardef$. 

Before proceeding with the proof, we state some basic definitions and results: we first focus on formally defining the Menger universal curve and some useful properties. We proceed by mentioning results on the dimension of the boundary, which, as mentioned in the introduction, relates to cohomological properties of the group. In particular, we will discuss the cohomological dimension of the group and its Dehn fillings, stating necessary results for the proof of Theorem \ref{maincorollary}.
\subsubsection{The Menger curve}
The Menger universal curve, also called Menger sponge, was defined by Menger in 1926 \cite{menger}. It is described by the following iterative procedure. Starting with a cube, the first step is to divide it into 27 equal cubes. The next step is to remove the middle cube of each face, as well as the one in the centre of the original cube. The result is 20 small cubes. Repeat the process for each of the 20 cubes, and continue iteratively. The first two repetitions of this process are illustrated in Figure (\ref{fig:menger}).

The Menger curve is \textit{universal}, meaning that every curve is homeomorphic to a subset of the Menger curve \cite{menger}. Here, curve means 1-dimensional compact metric space. In particular, every graph is homeomorphic to a subset of the Menger curve, so it includes a copy of the graphs $K_{3,3}$ and $K_5$ (Figure \ref{fig:K33}). This, in our case, implies that there exists a copy of $K_{3,3}$ and $K_5$ in $\bdry G$.

\begin{figure}
\centering
\scalebox{1.5}{
\begin{tikzpicture}
  {\input{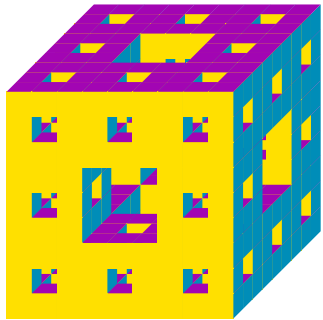}}
\end{tikzpicture}}
    \caption{The Menger sponge after 2 iterations.}
    \label{fig:menger}
\end{figure}
\subsubsection{Group cohomology and the dimension of the boundary} 
Before we proceed with proving Theorem \ref{maincorollary}, we first discuss the dimension of the boundary and how it relates to cohomological properties of the group. Corollary 1.4 of \cite{bestvinamess} implies that, if $G$ is a group with one-dimensional boundary, $H^k(G;\mathbb{Z}G)=0, \forall k\geq 3$. In particular, $\mathrm{ecd}(G)=2$, where $\mathrm{ecd}$ denotes the \textit{essential cohomological dimension}, in the sense of \cite{wang2021spectral}: the largest $k$ such that $H^k(G;\mathbb{Z}G)\neq 0$. Recall that, by \cite[Proposition 1.1]{bieri1978relative} we have a long exact sequence for the pair $(G,\Psub)$, where $\Psub$ is a collection of subgroups of $G$: \begin{equation}\label{leq1}
    ... \to H^{k-1}(\Psub;\mathbb{Z}G) \to H^k(G,\Psub;\mathbb{Z}G)\to H^k(G,\mathbb{Z}G)\to ...
\end{equation}
Viewing the group $\gbar$ as a Dehn filling of $G$, we can apply results on the cohomological dimension of Dehn fillings from \cite{wang2021spectral}. In particular, \cite[Corollary 1.4]{wang2021spectral} implies that, if $\mathrm{ecd(G,\Psub)}=2$, then $\mathrm{ecd}(\gbar,\hat{\Psub})\leq 2$, as long as $\gbar$ is a sufficiently long $F_\infty$ Dehn filling of the pair $(G,\Psub)$. We have the following long exact sequence for the Dehn filling pairs: \begin{equation}\label{leq2}
    \to ... H^{k-1}(\gbar,\hat{\Psub};\mathbb{Z}\gbar)\to H^k(\gbar;\mathbb{Z}\gbar) \to H^k(\hat{\Psub};\mathbb{Z}\gbar) \to ...
\end{equation}
We use this machinery and view the group pair $(\gbar,\hat{\Psub})$ as a Dehn filling of $(G,E(g))$, where $\hat{\Psub}$ consists of the subgroup $E(g)/\langle g^n\rangle$, allowing us to prove the following:
\begin{proposition}\label{cohomlemma}
    Let $G$ and $\gbar$ be hyperbolic groups as in Theorem \ref{maintheoremtorsion}. Then, if $\bdry G$ is one-dimensional, so is $\bdry \gbar$.
\end{proposition}
In proving this, we make use of the following facts:
\begin{claim}\label{claims}
    Let $G$ be a hyperbolic group satisfying the assumptions of Proposition \ref{cohomlemma}. Then the following hold:
    \begin{enumerate}
        \item $H^k(G,E(g);\mathbb{Z}G)=0$ for all $k\geq 3$.
        \item For any finite subgroup $F$ of $\gbar$, $H^k(F;\mathbb{Z}\gbar)=0$ for all $k \geq 3$.
    \end{enumerate}
\end{claim}
\begin{proof} Firstly, we note that, by \cite[Corollary 1.4]{bestvinamess}, $H^k(G;\mathbb{Z}G)=0$, because of the assumption that $\bdry G$ is one-dimensional.
    \begin{enumerate}
        \item By the long exact sequence in Equation (\ref{leq1}), letting $\Psub =E(g)$, it suffices to show that $H^{k-1}(E(g);\mathbb{Z}G)=0$. To see this, it is enough to show $H^{k-1}(E(g);\mathbb{Z}E(g))=0$; indeed, by \cite[Lemma 13.10.3]{topmethods}, if $H<G$ is a subgroup of type $F_\infty$ and $A$ is a free $\mathbb{Z}H$-module, $H^n(H;\mathbb{Z}H)=0$, implies $H^n(H;A)=0$ for any $n$. To use this result, note that $E(g)$ is of type $F_\infty$ since it is hyperbolic. Moreover, $\mathbb{Z}G$ is indeed a free $\mathbb{Z}E(g)$-module: this is because we can view $\mathbb{Z}G$ as the direct sum over all cosets of $E(g)$ in $G$: $\mathbb{Z}G=\oplus_{\mathrm{cosets}} \mathbb{Z}E(g)$. Now, to show that $H^{k-1}(E(g);\mathbb{Z}E(g))=0$, we use the fact that 
        the cohomology of a group and a finite index subgroup are isomorphic (see, for instance, \cite[Corollary 1.7.10]{loh2019group}). Since $E(g)$ is elementary, in particular it is virtually $\mathbb{Z}$, and $\mathbb{Z}$ has essential cohomological dimension one. Hence $H^k(E( g);\mathbb{Z}E(g))=0, \forall k\geq 2$, implying that $H^{k-1}(E(g);\mathbb{Z}G)=0$. Hence by Equation (\ref{leq1}), $H^k(G,E(g);\mathbb{Z}G)=0$, which proves the claim.
        \item  Since $F$ is a finite group, in particular it is of type $F_\infty$, and $\gbar$ is of type $F_\infty$ since it is hyperbolic. By the same argument as for the subgroup $E(g)$ above, we deduce that $H^k(F;\mathbb{Z}\gbar)=0, \forall k\geq 3$.
    \end{enumerate}
\end{proof}
\begin{proof}[Proof of Proposition \ref{cohomlemma}]
    By \cite[Corollary 1.4]{bestvinamess}, $\bdry G$ being one-dimensional implies that $\mathrm{ecd}(G)=2$. By the same result, in order to show that $\bdry \gbar$ is one-dimensional, it suffices to show that $\mathrm{ecd}(\gbar)=2$. By \cite[Corollary 1.4]{wang2021spectral}, we know that $\mathrm{ecd}(\gbar,\hat{\Psub})\leq 2$ for all sufficiently long $F_\infty$ Dehn fillings of $G$. By part (1) of Claim \ref{claims}, $H^k(\gbar,F;\mathbb{Z}\gbar)=0$. Therefore, in view of Equation (\ref{leq2}), it suffices to show that $H^k(F;\mathbb{Z}\gbar)=0$ for all $k \geq 3$, which is exactly the statement of part (2) of Claim \ref{claims}.
\end{proof}
\subsubsection{Hyperbolic groups with Menger curve boundary and the proof of Theorem \ref{maincorollary}}
In this subsection, we will combine Theorem \ref{maintheoremtorsion} and Proposition \ref{cohomlemma}, to show that, if we start with a one-ended hyperbolic group with Menger curve boundary, then the quotient group $\gbar$ also has Menger curve boundary. We first state a more general version which also holds for torsion groups. This result follows as a corollary of Theorem \ref{maintheoremtorsion}:
\begin{theorem}\label{maincortorsion}
    Let $G$ be a hyperbolic group with Menger curve boundary. Let $g\in G$ be an infinite order element.
    Then there exists $k> 0$ such that for all multiples $n$ of $k$, $\gbar = G\big/ \langle\langle g^n \rangle \rangle$ is a hyperbolic group with Menger curve boundary.

    If, in addition, $G$ is torsion-free, then the above holds for all large enough $n$.
\end{theorem}
\begin{proof}[Proof of Theorem \ref{maincortorsion}] 
Let $(Z,\rho)$ denote $\bdry G$ with the usual visual metric on the boundary (described in Section \ref{bdrymapsection}). Let $\gbar=G\big/ \langle\langle g^n\rangle\rangle$ for any large enough multiple $n$ of $k$, where $k$ is set as in Remark \ref{dehnsetup}. In the special case where $G$ is torsion-free, the proof below holds for all large enough $n$, by Remark \ref{dehnsetuptf}. \\\\
{\textit{Hyperbolicity and splittings.}} As seen in Theorem \ref{maintheoremtorsion}, $\gbar$ is hyperbolic. We claim that $\gbar$ admits no non-trivial elementary splittings. Since the Menger curve is connected and has no cut points (see, for instance, \cite[Theorem XII]{anderson1958}), then by Corollary \ref{gm7.9}, $G$ admits no elementary splittings. Applying Corollary \ref{nosplitbdry}, it follows that $\gbar$ also admits no elementary splittings.
\\\\
{\textit{Dimension of the boundary.}} Since $Z$ is a Menger curve, it is one-dimensional \cite{anderson1958}. By Proposition \ref{cohomlemma}, the boundary $\bdry \gbar$ is also one-dimensional.
\\\\
{\textit{Non-planarity.}} Finally, it remains to show that the boundary of $\gbar$ is non-planar. By Theorem \ref{maintheoremtorsion} and universality of the Menger curve, $\bdry \gbar$ contains a graph with $K_{3,3}$ minor. $K_{3,3}$ is non-planar, hence the graph with $K_{3,3}$ minor is itself non-planar by Kuratowski's theorem. Thus $\bdry \gbar$ is not planar by \cite[Result C]{claytor1934topological}, since it contains a non-planar graph. Note, the same argument holds if one uses $K_5$ rather than $K_{3.3}$. Since $\gbar$ has no elementary splittings and one-dimensional boundary, by \cite[Theorem 4]{kapovichkleiner}, $\bdry \gbar$ is a Menger curve. 
\end{proof}
\subsection{Non-planar boundaries}
Theorem \ref{maintheoremtorsion} has strong implications on the topology of the boundary of $\gbar$, since it shows that if a non-planar graph exists in the original boundary, then one exists in the quotient boundary. It turns out that this is a sufficient condition to assert non-planarity of the quotient boundary in a more general case than that of the Menger curve; even if the boundary is not one-dimensional, as long as the group $G$ is one-ended, the existence of a non-planar graph implies that the boundary $\bdry \gbar$ is non planar. 

As before, we state a version of Theorem \ref{generalnonplanar} that accounts for torsion groups:
\begin{corollary}[Corollary \ref{generalnonplanar}]\label{corctorsion}
Let $G$ be a one-ended hyperbolic group with non-planar boundary. Let $g\in G$ be an infinite order element. Then there exists $k> 0$ such that for all multiples $n$ of $k$, the hyperbolic group $\gbar = G\big/ \langle\langle g^n\rangle\rangle$ also has non-planar boundary.

If, in addition, $G$ is torsion-free, then the above holds for all large enough $n$.
\end{corollary}
\begin{proof}
    Since $G$ is one-ended, by \cite{bestvinamess, bowditch1999treelike, levitt1998non,swarup1996cut} $\bdry G$ is locally connected. In addition, by \cite{bowditch1998cut, swarup1996cut} $\bdry G$ has no (global) cut points. 

    Hence, we can apply the result of \cite[Page 835]{claytor1934topological}; in particular, this implies that, since $\bdry G$ is connected with no cut points, it is non-planar if and only if it contains a non-planar graph, say $\G$. 
    
    By Kuratowski's theorem, a graph is non-planar if and only if it contains either $K_{3,3}$ or $K_5$ as a minor. Hence, any graph containing $\G$ as a minor will also contain $K_{3,3}$ or $K_5$ as a minor, and will itself be non-planar. Therefore, we will show that $\bdry \gbar$ contains a graph with $\G$ as a minor, hence it contains a non-planar graph, and that it is connected with no cut points and no non-trivial elementary splittings. This will imply that it is itself non-planar by \cite[Page 835]{claytor1934topological}. 

    Applying Theorem \ref{maincortorsion} implies that there exists a graph $\G'$ which contains $\G$ as a minor. To see that $\bdry \gbar$ is also connected with no cut points, recall that by Corollary \ref{gm7.9}, $\bdry G$ has no non-trivial elementary splittings. Note that $G$ is not virtually Fuchsian since its boundary is non-planar by assumption. Hence $\bdry \gbar$ has no non-trivial elementary splittings by Corollary \ref{nosplitbdry}. In order to see that it has no cut points and is connected, we note that $\gbar$ is also not virtually Fuchsian; indeed, since $\bdry \gbar$ contains the non-planar graph $\G$, it cannot be homeomorphic to $\mathbb{S}^1$.

\end{proof}
\bibliography{references.bib}

@article{mackaysisto,
  title = {Quasi-Hyperbolic Planes in Relatively Hyperbolic Groups},
  author = {Mackay, John M. and Sisto, Alessandro},
  year = 2020,
  month = jan,
  journal = {Annales Academiae Scientiarum Fennicae Mathematica},
  volume = {45},
  number = {1},
  pages = {139--174},
  issn = {1239629X, 17982383},
  doi = {10.5186/aasfm.2020.4511},
  urldate = {2025-11-27},
  langid = {english}
}

@article{mackay,
  title = {Existence of Quasi-Arcs},
  author = {Mackay, John M.},
  year = 2008,
  month = jun,
  journal = {Proceedings of the American Mathematical Society},
  volume = {136},
  number = {11},
  eprint = {0708.0824},
  primaryclass = {math},
  pages = {3975--3981},
  issn = {0002-9939, 1088-6826},
  doi = {10.1090/S0002-9939-08-09444-6},
  urldate = {2025-12-29},
  archiveprefix = {arXiv},
  langid = {english}
}

@article{osin2006elementary,
  title={Elementary subgroups of relatively hyperbolic groups and bounded generation},
  author={Osin, Denis V.},
  journal={International Journal of Algebra and Computation},
  volume={16},
  number={01},
  pages={99--118},
  year={2006},
  publisher={World Scientific}
}

@article{anderson1958,
  title={One-dimensional continuous curves and a homogeneity theorem},
  author={Anderson, Richard D.},
  journal={Annals of Mathematics},
  volume={68},
  number={1},
  pages={1--16},
  year={1958},
  publisher={JSTOR}
}

@article{grovesmanning,
  title={Dehn fillings and elementary splittings},
  author={Groves, Daniel and Manning, Jason F.},
  journal={Transactions of the American Mathematical Society},
  volume={370},
  number={5},
  pages={3017--3051},
  year={2018}
}

@article{kapovichkleiner,
  title={Hyperbolic groups with low-dimensional boundary},
  author={Kapovich, Michael and Kleiner, Bruce},
  journal={Annales scientifiques de l'Ecole normale sup{\'e}rieure},
  volume={33},
  number={5},
  pages={647--669},
  year={2000}
}

@article{petrosyan2024cohomology,
  title={Cohomology of group theoretic {Dehn} fillings {II}},
  author={Petrosyan, Nansen and Sun, Bin},
  journal={Advances in Mathematics},
  volume={437},
  pages={109412},
  year={2024},
  publisher={Elsevier}
}

@article{bestvinamess,
  title={The boundary of negatively curved groups},
  author={Bestvina, Mladen and Mess, Geoffrey},
  journal={Journal of the American Mathematical Society},
  volume={4},
  number={3},
  pages={469--481},
  year={1991},
  publisher={JSTOR}
}

@book{bowditch1999treelike,
  title={Treelike structures arising from continua and convergence groups},
  author={Bowditch, Brian H.},
  journal={},
  volume={662},
  year={1999},
  publisher={American Mathematical Soc.}
}

@article{levitt1998non,
  title={Non-nesting actions on real trees},
  author={Levitt, Gilbert},
  journal={Bulletin of the London Mathematical Society},
  volume={30},
  number={1},
  pages={46--54},
  year={1998},
  publisher={Cambridge University Press}
}

@article{swarup1996cut,
  title={On the cut point conjecture},
  author={Swarup, Gadde},
  journal={Electronic Research Announcements of the American Mathematical Society},
  volume={2},
  number={2},
  pages={98--100},
  year={1996}
}

@article{mackaysisto2024maps,
  title={Maps between relatively hyperbolic spaces and between their boundaries},
  author={Mackay, John M. and Sisto, Alessandro},
  journal={Transactions of the American Mathematical Society},
  volume={377},
  number={02},
  pages={1409--1454},
  year={2024}
}

@article{sisto2012metric,
  title={On metric relative hyperbolicity},
  author={Sisto, Alessandro},
  journal={arXiv preprint arXiv:1210.8081},
  year={2012}
}

@article{claytor1934topological,
  title={Topological immersion of Peanian continua in a spherical surface},
  author={Claytor, Schiefelin},
  journal={Annals of Mathematics},
  volume={35},
  number={4},
  pages={809--835},
  year={1934},
  publisher={JSTOR}
}

@article{wang2021spectral,
  title={A spectral sequence for {Dehn} fillings},
  author={Wang, Oliver H.},
  journal={Algebraic \& Geometric Topology},
  volume={21},
  number={5},
  pages={2257--2272},
  year={2021},
  publisher={Mathematical Sciences Publishers}
}

@book{topmethods,
  title={Topological methods in group theory},
  author={Geoghegan, Ross},
  year={2008},
  publisher={Springer}
}

@article{loh2019group,
  title={Group cohomology},
  author={L{\"o}h, Clara},
  journal={Lecture notes http://www. mathematik. uni-regensburg. de/loeh/teaching/grouphom\_ss19},
  year={2019}
}

@article{bieri1978relative,
  title={Relative homology and {P{oincar{\'e}}} duality for group pairs},
  author={Bieri, Robert and Eckmann, Beno},
  journal={Journal of Pure and Applied Algebra},
  volume={13},
  number={3},
  pages={277--319},
  year={1978},
  publisher={North-Holland}
}

@article{delzant,
author = {Thomas Delzant},
title = {Sous-groupes distingu\'es et quotients des groupes hyperboliques},
volume = {83},
journal = {Duke Mathematical Journal},
number = {3},
publisher = {Duke University Press},
pages = {661 -- 682},
year = {1996},
doi = {10.1215/S0012-7094-96-08321-0},
URL = {https://doi.org/10.1215/S0012-7094-96-08321-0}
}

@article{gromov1987hyperbolic,
title={Hyperbolic Groups},
author={Gromov, Mikhael},
editor={Gersten, S. M.},
journal={Essays in Group Theory},
year={1987},
publisher={Springer New York},
pages={75--263},
}

@article{olshanskii,
  title={Almost every group is hyperbolic},
  author={Ol'Shanskii, Alexander Yu.},
  journal={International Journal of Algebra and Computation},
  volume={2},
  number={01},
  pages={1--17},
  year={1992},
  publisher={World Scientific Publishing Company}
}

@ARTICLE{dgp,
       author = {{Dahmani}, Fran{\c{c}}ois and {Guirardel}, Vincent and {Przytycki}, Piotr},
        title = {No-splitting property and boundaries of random groups},
      journal = {arXiv e-prints},
         year = {2009},
          doi = {10.48550/arXiv.0904.3854},
archivePrefix = {arXiv},
       eprint = {0904.3854}
}

@article{champetier1995proprietes,
  title={Propri{\'e}t{\'e}s statistiques des groupes de pr{\'e}sentation finie},
  author={Champetier, Christophe},
  journal={Advances in Mathematics},
  volume={116},
  number={2},
  pages={197--262},
  year={1995},
  publisher={New York: Academic Press, 1965-}
}

@article{danielski2022right,
  title={Right-angled Coxeter groups with Menger curve boundary},
  author={Danielski, Daniel},
  journal={Bulletin of the London Mathematical Society},
  volume={54},
  number={3},
  pages={977--995},
  year={2022},
  publisher={Wiley Online Library}
}

@article{Kuratowski1930,
author = {Kuratowski, Casimir},
journal = {Fundamenta Mathematicae},
language = {fre},
number = {1},
pages = {271-283},
title = {Sur le problème des courbes gauches en Topologie},
url = {http://eudml.org/doc/212352},
volume = {15},
year = {1930},
}

@article{benakliPolyedresAvecNombre1993,
  title = {Poly{\`e}dres Avec Un Nombre Fini de Types de Sommets},
  author = {Benakli, Nadia},
  year = 1993,
  journal = {Comptes rendus de l'Acad{\'e}mie des sciences. S{\'e}rie 1, Math{\'e}matique},
  volume = {317},
  number = {9},
  pages = {863--866},
  publisher = {Elsevier},
  address = {Paris},
  issn = {0764-4442},
}

@article{dymaraBoundariesRightangledHyperbolic2007,
  title = {Boundaries of Right-Angled Hyperbolic Buildings},
  author = {Dymara, Jan and Osajda, Damian},
  year = 2007,
  journal = {Fundamenta Mathematicae},
  volume = {197},
  pages = {123--165},
  issn = {0016-2736, 1730-6329},
  doi = {10.4064/fm197-0-6},
  urldate = {2026-06-12},
  langid = {english}
}

@article{bowditch1998cut,
  title={Cut points and canonical splittings of hyperbolic groups},
  author={Bowditch, Brian H.},
  journal={Acta Math},
  volume={180},
  number={145},
  pages={186},
  year={1998}
}

@article{bowditch2012relatively,
  title={Relatively hyperbolic groups},
  author={Bowditch, Brian H.},
  journal={International Journal of Algebra and Computation},
  volume={22},
  number={03},
  pages={1250016},
  year={2012},
  publisher={World Scientific}
}

@article{grovesDehnFillingRelatively2008,
  title = {Dehn Filling in Relatively Hyperbolic Groups},
  author = {Groves, Daniel and Manning, Jason F.},
  year = 2008,
  month = nov,
  journal = {Israel Journal of Mathematics},
  volume = {168},
  number = {1},
  pages = {317--429},
  issn = {1565-8511},
  doi = {10.1007/s11856-008-1070-6},
  urldate = {2026-07-22},
  langid = {english}
}

@article{osinPeripheralFillingsRelatively2007,
  title = {Peripheral Fillings of Relatively Hyperbolic Groups},
  author = {Osin, Denis V.},
  year = 2007,
  month = feb,
  journal = {Inventiones mathematicae},
  volume = {167},
  number = {2},
  pages = {295--326},
  issn = {1432-1297},
  doi = {10.1007/s00222-006-0012-3},
  urldate = {2026-07-22},
  langid = {english}
}

@article{drutusapir,
  title = {Tree-Graded Spaces and Asymptotic Cones of Groups},
  author = {Dru{\c t}u, Cornelia and Sapir, Mark},
  year = 2005,
  month = sep,
  journal = {Topology},
  volume = {44},
  number = {5},
  pages = {959--1058},
  issn = {00409383},
  doi = {10.1016/j.top.2005.03.003},
  urldate = {2026-09-09},
  copyright = {https://www.elsevier.com/tdm/userlicense/1.0/},
  langid = {english}
}

@article{bonkschramm,
  title = {Embeddings of {{Gromov}} Hyperbolic Spaces},
  author = {Bonk, Mario and Schramm, Oded},
  year = 2000,
  month = jun,
  journal = {Geometric \& Functional Analysis GAFA},
  volume = {10},
  number = {2},
  pages = {266--306},
  issn = {1420-8970},
  doi = {10.1007/s000390050009},
  urldate = {2026-09-14},
  langid = {english}
}

@article{menger,
 author = {Menger, Karl},
 title = {Allgemeine {R{\"a}ume} und {Cartesische} {R{\"a}ume}. {II}: {{\"U}ber} umfassendste {{\(n\)}}-dimensionale {Mengen}.},
 fjournal = {Proceedings. Akadamie van Wetenschappen Amsterdam},
 journal = {Proc. Akad. Wet. Amsterdam},
 issn = {0370-0348},
 volume = {29},
 pages = {1125--1128},
 year = {1926},
 language = {German},
 zbMATH = {2587559},
 JFM = {52.0595.01}
}
\bibliographystyle{alpha}
\end{document}